\documentclass[fleqn, 12pt]{amsart}
\usepackage{import}

\usepackage{amsmath} % for math enivronments
\usepackage{fancyhdr} % for customizable headers/footers
\usepackage{amsthm} % defines additional math environments
\usepackage{amsfonts} % math fonts
\usepackage{amssymb} % math symbols
\usepackage{mathrsfs} % more math fonts
\usepackage[mathscr]{euscript} % math fonts
\usepackage[marginratio=1:2,paper=letterpaper,tmargin=52pt,lmargin=74pt,rmargin=74pt,includehead]{geometry} % specifies page layout
\usepackage[all]{xy} % for (exact) sequences
\usepackage{enumerate} % allows for customization when using enumerate
\usepackage[euler-digits,euler-hat-accent]{eulervm} % euler font
\usepackage{hyperref} % references
\usepackage{stmaryrd} % double slash
\usepackage[dvipsnames]{xcolor} % colored fonts
\usepackage{graphicx} % for images
\usepackage{tikz-cd} % commutative diagrams
\usepackage{mathtools} % over/underbrackets

\xyoption{rotate}

\hypersetup{colorlinks = true, linkcolor=blue}

\newcommand{\hyp}[1]{\hyperref[#1]{\ref*{#1}}} % reference

\newcommand{\bi}[1]{\textbf{\textit{#1}}} % bold italics
\newcommand{\ite}{\begin{itemize}}
\newcommand{\eti}{\end{itemize}}

\newcommand{\C}{\mathbb{C}} % complex numbers
\newcommand{\Q}{\mathbb{Q}} % rational numbers
\newcommand{\R}{\mathbb{R}} % real numbers
\newcommand{\Z}{\mathbb{Z}} % integers

\newcommand{\calB}{\mathcal{B}}
\newcommand{\calC}{\mathcal{C}}

\newcommand{\calI}{\mathcal{I}}

\newcommand{\calL}{\mathcal{L}}
\newcommand{\calM}{\mathcal{M}}

\newcommand{\calS}{\mathcal{S}}

\newcommand{\calX}{\mathcal{X}}

\newtheorem{thm}{Theorem}[section]
\newtheorem{cor}[thm]{Corollary}
\newtheorem{lem}[thm]{Lemma}
\newtheorem{prop}[thm]{Proposition}
\theoremstyle{definition}
\newtheorem{ex}[thm]{Example}
\newtheorem{clm}{Claim}

\begin{document}

\title{Algebraic Intersection Spaces}
\date{\today}
\author{Christian Geske}
\address{Department of Mathematics,
          University of Wisconsin-Madison,
          480 Lincoln Drive, Madison WI 53706-1388, USA.}
\email{cgeske@wisc.edu}
\thanks{I thank my advisor Larentiu Maxim for helpful discussion and guidance and J\"{o}rg Sch\"{u}rmann for helpful suggestions regarding selection of tubular neighborhoods.
I would also like to thank the anonymous referee for his insight and constructive suggestions.
C. Geske gratefully acknowledges the support provided by the NSF-RTG grant \#1502553 at the University of Wisconsin-Madison.}
\begin{abstract}
We define a variant of intersection space theory that applies to many compact complex and real analytic spaces $X$, including all complex projective varieties; this is a significant extension to a theory which has so far only been shown to apply to a particular subclass of spaces with smooth singular sets.
We verify existence of these so-called \textit{algebraic} intersection spaces and show that they are the (reduced) chain complexes of known topological intersection spaces in the case that both exist.   
We next analyze ``local duality obstructions'', which we can choose to vanish, and verify that algebraic intersection spaces satisfy duality in the absence of these obstructions.
We conclude by defining an untwisted algebraic intersection space pairing, whose signature is equal to the Novikov signature of the complement in $X$ of a tubular neighborhood of the singular set.       
\end{abstract}
\maketitle

\thispagestyle{fancy}
\renewcommand{\headrulewidth}{0pt}
\setlength{\footskip}{20pt}
\rhead{}
\cfoot{}

\section{Introduction} 

Singular complex varieties typically lack the Poincar\'{e} duality enjoyed by their nonsingular counterparts.  
Approaches to rectifying this disparity can be found in Cheeger's $L^2$-cohomology and Goresky-MacPherson's middle-perversity intersection homology \cite{GM1}, \cite{GM2}.
These theories endow both singular and nonsingular spaces with an intrinsic duality, and moreover exhibit stability under small resolutions, but suffer from instability under smooth deformations.
A natural question is: does a duality-satisfying homology-type theory exist that behaves well under smooth deformations?
This question has been broached and, for hypersurfaces with isolated singularities, answered partially in the affirmative in \cite{BaMa1} and \cite{BaMa2}, in which an alternate theory is utilized: the intersection space homology theory introduced in \cite{Ba1}.
At this point, a universal generalization has not been discovered, because intersection space theory has not been defined for the vast majority of singular spaces.   
This paper devises a variant/extension of intersection space theory applicable to all complex projective varieties, and so enables future endeavors in this and other directions.  

Given a perversity $\bar{p}$, in the sense of intersection homology theory, Banagl in \cite{Ba1} associated to certain real $n$-dimensional stratified pseudomanifolds $X$ CW-complexes $I^{\bar{p}}X$, the \bi{perversity $\bar{p}$ intersection spaces}.
When $X$ is compact and oriented, there exist duality isomorphisms:
\begin{align*}
\tilde{H}^*(I^{\bar{p}}X;\Q) \cong \tilde{H}_{n-*}(I^{\bar{q}}X;\Q)
\end{align*}      
where $\bar{q}$ is the complementary perversity to $\bar{p}$.
The intersection space construction has been shown to apply when $X$ has isolated singularities, under certain conditions when $X$ has depth one (see \cite{BaC} for the latter), and for arbitrary depth spaces under certain restrictive conditions on link bundles of strata (see \cite{AF}).
The resulting homology and cohomology vector spaces depend on the stratification of $X$, and in the non-isolated setting can depend on a ``local'' choice made in the intersection space construction.   
Special cases where the stratification of $X$ is more elaborate have been studied, for example in \cite{Ba3}, but no all-encompassing picture has been painted.
Despite the limited collection of spaces for which it is defined, intersection space theory has had applications in multiple fields: fiber bundle theory \cite{Ba2}, algebraic geometry and smooth deformations \cite{BaMa1} and \cite{BaMa2}, perverse sheaves \cite{BaMa3}, and theoretical physics \cite[Chapter 3]{Ba1}.
    
This paper selects as its target the collection of compact orientable Whitney stratified pseudomanifolds $X$ which are subvarieties of a real/complex analytic manifold, and from them systematically extracts chain complex alternatives to topological intersection spaces that we equip with the moniker perversity $\bar{p}$ \bi{algebraic intersection spaces}.   
Though deprived of a topology, these algebraic intersection spaces carry homology, which we require to be an extension to, not replacement of, the already existing intersection space homology.

The introduction is followed in Section \hyp{Motivation} by a discussion of the motivation for the algebraic intersection space construction, replete with technical remarks on the features of the present theory that make it difficult to generalize. 
This section is most suited to a reader familiar with the current theory and is not necessary to understand the rest of the paper.

Section \hyp{Preliminary} collects general theorems and lemmas that will enable, or in some cases merely streamline, the arguments to be made toward the main results of the paper. Except perhaps for Subsection \hyp{Tubular Neighborhoods} on tubular neighborhoods of singular sets, detailed reading of this section should be left until the corresponding theorem or lemma is referred to in a proof from the final three sections, which form the core of the paper.    

Just as with intersection spaces, algebraic intersection spaces are built in parts: first locally, then globally. 
Section \hyp{LocalConstruction} describes the local construction, which rests on a \bi{local intersection approximation}.
Its concluding subsection \hyp{LocalDualityIsomorphism} describes \bi{local duality obstructions}, an unfortunate feature of certain local intersection approximations that precludes a global duality.
On the bright side, there always \textit{exist} local intersection approximations for which there are no local duality obstructions, a statement proven in the very same subsection. 

Section \hyp{GlobalConstruction} takes the local construction and converts it into a global algebraic intersection space associated to $X$.
In the case that the local duality obstructions vanish, duality isomorphisms are constructed between the homology and cohomology of complementary perversity algebraic intersection spaces.

Section \hyp{Example} explicitly constructs a topological intersection space for a depth two pseudomanifold.
It is compared against the \cite{AF} construction, from which it is shown to differ on the level of homology.

When $X$ is a Witt space, Section \hyp{IntersectionSpacePairing} extracts the signature of an \bi{intersection space pairing} on the middle dimensional homology of the algebraic intersection space (again in the case that the local duality obstructions vanish), which turns out to be equal to the Novikov signature of $X$ minus an open tubular neighborhood of the singular set. 

\section{Motivation}\label{Motivation}
For this discussion, suppose $X$ is a compact real pseudomanifold with singular set $\Sigma$.
Moreover, assume that $\Sigma$ has a nice closed ``tubular neighborhood'' $T$ in $X$ (to be defined later). 
Let $\bar{p}$ denote a perversity function.

In the case that $\Sigma$ is a discrete set of points, the boundary $\partial T$ is the union of links of singular points, and Banagl in \cite{Ba1} introduced the notion of a topological intersection space $I^{\bar{p}}X$ associated to $X$. 
The key component to the construction was a local CW complex \textit{Moore approximation} to $\partial T$, which in effect provided a topological ``splitting'' of the natural chain map $C_\bullet(\partial T) \rightarrow IC^{\bar{p}}_\bullet(T)$: the Moore approximation is a map $(\partial T)_{<m} \rightarrow \partial T$ for which the composition $H_*(\partial T)_{<m} \rightarrow H_*(\partial T) \rightarrow IH^{\bar{p}}_*(T)$ is an isomorphism, where here $m$ depends on $\bar{p}$ and the dimension of $X$. 
Banagl developed a ``truncation tool'' which took a simply-connected link and detailed how to obtain a Moore approximation.
Banagl also showed how to construct intersection spaces for depth one spaces whose link bundles are trivial \cite[Section 2.9]{Ba1}.

In \cite{BaC} the theory was extended to a larger class of $X$, allowing the singular set to have arbitrary dimension, but requiring it be smooth and for $\partial T$ to have the structure of a fiber bundle over $\Sigma$.
The local topological tool developed was \textit{equivariant Moore approximation}, a map $ft_{<m}\partial T \rightarrow \partial T$ giving an isomorphism $H_*(ft_{<m}\partial T) \rightarrow IH_*(T)$, where here $m$ depends on $\bar{p}$ and the codimension of $\Sigma$.
While of interest in its own right, the tool need not apply to every space from the class just described. 
With so many roadblocks and restrictions popping up even for this relatively pleasant class of spaces, it seemed unlikely that a full generalization was near at hand.  

Our eyes were set on an assumption that appears in the generalization described in the previous paragraph: that the map $H_*(\partial T) \rightarrow IH^{\bar{p}}_*(T)$ be \textit{surjective}.
This certainly fails to hold in general, and fails more often the more complex the stratification becomes (in the discrete case this \textit{never} fails). 
We circumvent it by working with the image of $H_*(\partial T) \rightarrow IH_*^{\bar{p}}(T)$, rather than $IH_*^{\bar{p}}(T)$ itself.

\section{Preliminary Results}\label{Preliminary}

\subsection{Triangulated Categories}

The set of tools consisting of the language and results of triangulated categories will streamline a number of the arguments made throughout this paper.
We draw from \cite{W} throughout this subsection.

\begin{lem}\label{indexacttri}
Suppose $R$ is a commutative, unital ring and:
\begin{align*}A_\bullet \xrightarrow{f_\bullet} B_\bullet \xrightarrow{g_\bullet} C_\bullet \xrightarrow{-1}\end{align*}
is an exact triangle of chain complexes of $R$-modules.
For $i \in \Z$ set:
\begin{align*}
Z_i &= im \left[H_i(A_\bullet) \xrightarrow{f_*} H_i(B_\bullet)\right]\\
Y_i &= coker \left[H_i(B_\bullet) \xrightarrow{g_*} H_i(C_\bullet)\right].
\end{align*}
Interpret $H_\bullet(A_\bullet)$, $Z_\bullet$, and $Y_\bullet$ as chain complexes with zero differential.
Then there is an exact triangle:
\begin{align*}H_\bullet(A_\bullet) \rightarrow Z_\bullet \rightarrow Y_\bullet \xrightarrow{-1}\end{align*}
whose maps are those induced by the maps of the long exact sequence in homology associated to the given exact triangle.  
\end{lem} 
\begin{proof}
The given exact triangle induces a long exact sequence in homology, which provides us the natural identification:
\begin{align*}Y_i = ker \left[H_{i-1}(A_\bullet) \xrightarrow{f_*} H_{i-1}(B_\bullet)\right].\end{align*}
So there is a short exact sequence of chain complexes:
\begin{align*}0 \rightarrow Y_{\bullet+1} \rightarrow H_\bullet(A_\bullet) \rightarrow Z_\bullet \rightarrow 0.\end{align*} 
By \cite{W} Example 10.4.9, this induces an exact triangle:
\begin{align*}Y_{\bullet+1} \rightarrow H_\bullet(A_\bullet) \rightarrow Z_\bullet \xrightarrow{-1}.\end{align*}
By the second axiom of triangulated categories, stated in \cite{W}, this new triangle induces by translation another exact triangle:
\begin{align*}H_\bullet(A_\bullet) \rightarrow Z_\bullet \rightarrow Y_\bullet \xrightarrow{-1}.\end{align*}
\end{proof}

\noindent A similar statement for cochain complexes is given below.

\begin{lem}\label{coindexacttri}
Suppose $R$ is a commutative, unital ring and:
\begin{align*}C^\bullet \xrightarrow{f_\bullet} B^\bullet \xrightarrow{g_\bullet} A^\bullet \xrightarrow{+1}\end{align*}
is an exact triangle of cochain complexes of $R$-modules.
For $i \in \Z$ set:
\begin{align*}
Z^i &= coim \left[H^i(B^\bullet) \xrightarrow{g_*} H^i(A^\bullet)\right]\\
Y^i &= ker \left[H^i(C^\bullet) \xrightarrow{f_*} H^i(B^\bullet)\right].
\end{align*}
Interpret $H^\bullet(A^\bullet)$, $Z^\bullet$, and $Y^\bullet$ as cochain complexes with zero differential.
Then there is an exact triangle:
\begin{align*}Y^\bullet \rightarrow Z^\bullet \rightarrow H^\bullet(A^\bullet) \xrightarrow{+1}\end{align*}
whose maps are those induced by the maps of the long exact sequence in cohomology associated to the given exact triangle.   
\end{lem}
\begin{proof}
The first isomorphism theorem gives a natural identification:
\begin{align*}Z^i = im\,\left[H^i(B^\bullet) \xrightarrow{g_*} H^i(A^\bullet)\right].\end{align*}
The long exact sequence in cohomology of the given exact triangle provides further identifications:
\begin{align*}
Z^i &= ker \left[H^i(A^\bullet) \rightarrow H^{i+1}(C^\bullet)\right]\\
Y^i &= im \left[H^{i-1}(A^\bullet) \rightarrow H^i(C^\bullet)\right].
\end{align*}
So there is a short exact sequence of cochain complexes:
\begin{align*}0 \rightarrow Z^\bullet \rightarrow H^\bullet(A^\bullet) \rightarrow Y^{\bullet+1} \rightarrow 0.\end{align*}
Proceed with the same argument given in the proof of Lemma \hyp{indexacttri} to obtain the desired exact triangle.
\end{proof}

\subsection{Tubular Neighborhoods of Singular Sets}\label{Tubular Neighborhoods}
Throughout this subsection, let $\calC$ denote the category of real subanalytic sets (we refer to \cite{VM} for a description of this category and its properties).
Let $X$ denote a compact subvariety of a real analytic manifold and let $\Sigma \subset X$ denote its singular set.

\begin{lem}\label{zeroset}
There exists a $\calC$-map $f: X \rightarrow \R_{\geq 0}$ such that $f^{-1}0 = \Sigma$.
\end{lem}
\begin{proof}
Let $M$ be the real analytic manifold containing $X$.
\cite{VM} D.19 provides a $\calC$-map $M \rightarrow \R$ with zero set $\Sigma$.
Square this map and restrict it to $X$ to obtain the desired $\calC$-map. This argument does not require that $X$ be compact.	
\end{proof}

Suppose $f$ is a map satisfying the conditions of Lemma \hyp{zeroset}. 
Because $X$ is compact, $f$ is proper and therefore can be smoothly Whitney stratified into subanalytic sets (see \cite{VM} 1.19 and the following remark).
Let $\calS$ denote such a stratification of $X$ and $\calS'$ the stratification of $\R_{\geq 0}$.
Because $0$-dimensional subanalytic sets are discrete (see \cite{VM} 1.15)
there exists a minimal $\epsilon_0 > 0$ for which $\{\epsilon_0\} \in \calS'$.

A triple $\xi = (f, \calS, \epsilon_0)$ as in the previous paragraph is called \bi{global tubular data} for the singular set of $X$.
Given such data and $0 < \epsilon < \epsilon_0$ we let $T = T(\epsilon)$ denote $f^{-1}[0,\epsilon]$ and call it a \bi{(closed) tubular neighborhood} of $\Sigma$ in $X$ associated to $\xi$.
$T$, its boundary $\partial T$, and its interior $T^\circ$ will always be equipped with Whitney stratifications induced by $\calS$ (see \cite{GM3} I.1.3.1).

\begin{lem}\label{strathomeo}
Let $\xi$ be tubular data.
Then:
\begin{enumerate}[(i)]
\item the stratified homeomorphism type of $T(\epsilon)$ does not depend on $\epsilon$.
\item the inclusion $\partial T(\epsilon) \hookrightarrow T(\epsilon)-\Sigma$ is a codimension preserving stratified homotopy equivalence.
\item for $0 < \epsilon' < \epsilon < \epsilon_0$ the inclusions $T(\epsilon')^\circ \hookrightarrow T(\epsilon)^\circ$ and $T(\epsilon')^\circ - \Sigma \hookrightarrow T(\epsilon)^\circ - \Sigma$ are stratified homotopy equivalences. 
\end{enumerate}
\end{lem}
\begin{proof}
By Thom's isotopy lemma \cite{Di} Theorem 1.3.5 and contractibility of the range, the restriction $f^{-1}(0,\epsilon_0) \rightarrow (0, \epsilon_0)$ is a trivial stratified fiber bundle.
We can use this trivialization to construct a stratified homeomorphism $T(\epsilon') \cong T(\epsilon)$ for any $0 < \epsilon' < \epsilon < \epsilon_0$ and to show that the inclusions in question are codimension preserving stratified homotopy equivalences (see \cite{Fr} Definition 2.9.10).
\end{proof}

We say that tubular data $\xi$ is \bi{pseudomanifold compatible (pc)} if and only if $\calS \in \xi$ induces on $X$ the structure of a pseudomanifold.
In other words: the top dimensional strata of $\calS$ are dense and the remaining strata have codimension at least two. 

\textit{Remark.} Tubular data always exists.
Pc tubular data $\xi$ exists for example if $X$ is complex and equidimensional. 
If $\xi$ is pc, then $T$ inherits the structure of a pseudomanifold with boundary, since $\partial T$ admits a collar neighborhood (see \cite{Fr} Definition 2.7.1) by Thom's isotopy lemma.  

\subsection{Intersection Homology and Cohomology}
Let $X$ denote a compact subvariety of a real analytic manifold.
Assume $X$ admits pc tubular data (e.g.\ $X$ is complex and equidimensional).
All tubular data in this subsection is assumed to be pc.

We implicitly draw from \cite{Fr} throughout this subsection.
For basic definitions see for example \cite{Fr} Definition 3.1.4, Remark 3.1.5, and Definition 3.4.1.

We begin by proving that the intersection homologies/cohomologies of a tubular neighborhood $T$ are independent of choices.

\begin{lem}[Invariance]\label{invariance}
The intersection homologies and cohomologies of the tubular neighborhood $T$, its boundary $\partial T$, and the pair $(T,\partial T)$ do not depend on any choices.  	
\end{lem}
\begin{proof}
By \cite{Fr} Corollary 4.1.11 codimension preserving stratified homotopy equivalences preserve intersection homologies and cohomologies.
Moreover, the intersection homologies of a pseudomanifold with boundary are naturally isomorphic to those of the pseudomanifold minus the boundary.
Together with Lemma \hyp{strathomeo} (ii) we then reduce to checking invariance for $T^\circ$, $T^\circ - \Sigma$, and the pair they form.
By Lemma \hyp{strathomeo} (iii) the choice of $\epsilon$ for fixed $\xi$ doesn't affect these homologies and cohomologies.
 
Let $\xi$ and $\widetilde{\xi}$ be pc tubular data for $X$.
Because $X$ is compact, any tube for one data contains a smaller tube for the other data.
In particular we can construct a sequence of tubes:
\begin{align*}
T_0 \subset \widetilde{T} \subset T \subset \widetilde{T}_0 	
\end{align*}
corresponding in an alternating fashion to $\xi$ and $\widetilde{\xi}$.
We may restrict these inclusions to the open tubes $(-)^\circ$ or to the open tubes minus the singular set $(-)^\circ - \Sigma$.
Any composition of two of these restricted inclusions by Lemma \hyp{strathomeo} (iii) induces isomorphisms on intersection homologies and cohomologies.
A simple argument shows that the same is true for the central inclusions, namely in the commutative diagram:
$$\begin{tikzcd}
T^\circ - \Sigma \ar[hookrightarrow]{r} & T^\circ\\
\widetilde{T}^\circ - \Sigma \ar[hookrightarrow]{u} \ar[hookrightarrow]{r} & \widetilde{T}^ \circ \ar[hookrightarrow]{u}	
\end{tikzcd}$$ 
the vertical maps induces isomorphisms on intersection homologies and cohomologies.
From naturality of long exact sequences of pairs and the five lemma, we obtain also an isomorphism for intersection homologies and cohomologies of the pair.
We have successfully compared the tubular data $\xi$ and $\widetilde{\xi}$.  
\end{proof}

\textit{Remark.} Our definition of tubular neighborhood is actually unnecessarily restrictive.
If $T(\epsilon)$ for $0 < \epsilon < \epsilon_0$ is any increasing family of closed neighborhoods of $\Sigma$ such that:
\begin{enumerate}[(i)]
	\item $\bigcap_{\epsilon > 0} T(\epsilon) = \Sigma$
	\item the $T(\epsilon)$ are pseudomanifolds whose boundary $\partial T(\epsilon)$ is a submanifold of $X-\Sigma$
	\item  the $T(\epsilon)$ satisfy the conditions of Lemma \hyp{strathomeo}
\end{enumerate}
then such a $T(\epsilon)$ will do equally well, and will not affect the validity of Lemma \hyp{invariance}.
We will use such a tubular neighborhood in Section \hyp{Example}.

\medskip

The duality captured by any intersection space construction is inseparable from a Lefschetz duality described by the results of intersection homology and cohomology. 
To this end we will need orientability.
If $X$ as a pseudomanifold is oriented, then a tubular neighborhood $T$ inherits an orientation from $X$.  

\begin{lem}\label{tubneighdual}
Let $(\bar{p}, \bar{q})$ be complementary perversity functions, $A$ an abelian group, and $T$ a tubular neighborhood of $\Sigma$.
Then there is an exact triangle of (intersection) chain complexes:
\begin{align*}C_\bullet(\partial T; A) \rightarrow IC^{\bar{p}}_\bullet(T; A) \rightarrow IC^{\bar{p}}_\bullet(T,\partial T; A) \xrightarrow{-1}\end{align*}
and of (intersection) cochain complexes:
\begin{align*}IC_{\bar{q}}^\bullet(T,\partial T;A) \rightarrow IC_{\bar{q}}^\bullet(T;A) \rightarrow C^\bullet(\partial T;A) \xrightarrow{+1}.\end{align*}
If $X$ is oriented of dimension $n$, then there is a natural duality isomorphism between their (shifted) associated long exact sequences:
$$\begin{tikzcd}
\cdots \ar[r] & H_i(\partial T;A) \ar[d, "D", "\simeq" {anchor=south, rotate=90, inner sep=1mm}] \ar[r] & IH^{\bar{p}}_i(T;A) \ar[d, "D", "\simeq" {anchor=south, rotate=90, inner sep=1mm}] \ar[r] & IH^{\bar{p}}_i(T,\partial T;A) \ar[d, "D", "\simeq" {anchor=south, rotate=90, inner sep=1mm}] \ar[r] & \cdots\\
\cdots \ar[r] & H^{n-i-1}(\partial T;A) \ar[r] & IH_{\bar{q}}^{n-i}(T,\partial T;A) \ar[r] & IH_{\bar{q}}^{n-i}(T;A) \ar[r] & \cdots  
\end{tikzcd}$$
\end{lem}
\begin{proof}
Because $\partial T$ is non-singular (recall the definition of tubular neighborhoods) there are quasi-isomorphisms $IC^{\bar{p}}_\bullet(\partial T; A) \simeq C_\bullet(\partial T; A)$ and $IC_{\bar{q}}^\bullet(\partial T; A) \simeq C^\bullet(\partial T;A)$.
We may therefore replace the former with the latter when we only care about complexes up to quasi-isomorphism.

By the definition of the relative intersection chain complex given in \cite{Fr} Definition 4.3.7, there is a short exact sequence of chain complexes:
\begin{align*}0 \rightarrow IC^{\bar{p}}_\bullet(\partial T;A) \rightarrow IC^{\bar{p}}_\bullet(T;A) \rightarrow IC^{\bar{p}}_\bullet(T,\partial T;A) \rightarrow 0.\end{align*}
This (together with the first paragraph) produces the first distinguished triangle.
The second is obtained analogously.
The duality isomorphism between their shifted long exact sequences is described in the proof of \cite{Fr} Corollary 8.3.10.

\end{proof}

\subsection{Linear Algebra}
We will exclusively use field coefficients for the main results of the paper.
To this end, we let $k$ denote a field and establish a few lemmas.

\begin{lem}\label{dualmap}
Suppose $f: A \rightarrow B$ is a morphism of $k$-vector spaces with dual map $f^*: B^* \rightarrow A^*$.
Then there are natural identifications:
\begin{align*}
&\left(coker\,f\right)^* = ker\,(f^*)\\
&\left(im\,f\right)^* = coim\,(f^*).
\end{align*} 
\end{lem}
\begin{proof}
We dualize the exact sequence:
\begin{align*}A \xrightarrow{f} B \rightarrow coker\, f \rightarrow 0\end{align*}
to obtain exact:
\begin{align*}0 \rightarrow (coker\,f)^* \rightarrow B^* \xrightarrow{f^*} A^*,\end{align*}
which proves the first identification by showing that $(coker\,f)^*$ maps isomorphically onto $ker\,(f^*)$.
We prove the second identification by dualizing the exact sequence:
\begin{align*}0 \rightarrow im\,f \rightarrow B \rightarrow coker\,f \rightarrow 0\end{align*}
to obtain exact:
\begin{align*}0 \rightarrow (coker\,f)^* \rightarrow B^* \rightarrow (im\,f)^* \rightarrow 0\end{align*} and utilizing the first identification $(coker\, f)^* = ker\,(f^*)$.
\end{proof}

\begin{lem}\label{dualconstruction}
Suppose there is a commutative diagram of exact sequences of $k$-vector spaces:
$$
\begin{tikzcd}
\cdots \ar[r] & C_{i+1}  \ar[r, "\partial_{i+1}"] & A_i \ar[r, "g_i"] & B_i \ar[r, "h_i"] & C_i \ar[r, "\partial_i"] & A_{i-1} \ar[r] & \cdots\\
\cdots \ar[r] & F_{i+1} \ar[u, "D''_{i+1}", "\simeq" {anchor=south, rotate=-90, inner sep=1mm}]  \ar[r, "\delta_{i+1}"] & D_i \ar[u, "D'_i", "\simeq" {anchor=south, rotate=-90, inner sep=1mm}] \ar[r, "u_i"] & E_i \ar[r, "v_i"] & F_i \ar[u, "D''_i", "\simeq" {anchor=south, rotate=-90, inner sep=1mm}] \ar[r, "\delta_i"] & D_{i-1} \ar[u, "D'_{i-1}", "\simeq" {anchor=south, rotate=-90, inner sep=1mm}] \ar[r] & \cdots .
\end{tikzcd}
$$
Then for each pair $(r_\bullet, s_\bullet)$, where $r_\bullet: im\, h_\bullet \rightarrow B_\bullet$ and $s_\bullet: im\, v_\bullet \rightarrow E_\bullet$ are (families) of sections, 
there exists an induced isomorphism:
\begin{align*} D_\bullet = D^{(r_\bullet, s_\bullet)}_\bullet: E_\bullet \xrightarrow{\simeq} B_\bullet \end{align*}
whose description is found in the proof.
\end{lem}
\begin{proof}
This is the content of \cite[Lemma 2.46]{Ba1}.
We recreate the argument here, because it is important in Section \hyp{IntersectionSpacePairing} to understand exactly how $D_\bullet^{(r_\bullet, s_\bullet)}$ relates to the the choice of $(r_\bullet, s_\bullet)$.
  
Because the diagram commutes and the rows are exact, there is an induced diagram of exact sequences:
$$
\begin{tikzcd}
0 \ar[r] & \overbrace{coker\, \partial_{i+1}}^{\simeq ker\, h_i} \ar[r, "\bar{g}_i"] & B_i \ar[r, "h_i"] & \overbrace{ker\, \partial_i}^{im\, h_i} \ar[r] & 0 \\
0 \ar[r] & \underbrace{coker\, \delta_{i+1}}_{\simeq ker\, k_i} \ar[u, "D'_i", "\simeq" {anchor=south, rotate=-90, inner sep=1mm}] \ar[r, "\bar{u}_i"] & E_i \ar[r, "v_i"] & \underbrace{ker\, \delta_i}_{im\, v_i} \ar[u, "D''_i", "\simeq" {anchor=south, rotate=-90, inner sep=1mm}] \ar[r] & 0 
\end{tikzcd}
$$
where we are abusing notation by allowing $D'_i$ and $D''_i$ to denote induced isomorphisms.
The sections $r_i$ and $s_i$ give splittings
\begin{align*}
&B_i = im\, \bar{g}_i \oplus im\, r_i = im\, g_i \oplus im\, r_i\\
&E_i = im\, \bar{u}_i \oplus im\, s_i = im\, u_i \oplus im\, s_i.
\end{align*}
The isomorphism $D_i: E_i \rightarrow B_i$ induced by the splittings $(r_\bullet, s_\bullet)$ is described on components as follows.
\begin{align*}
&D_i(u_i\alpha) = g_iD'_i(\alpha), ~ \alpha \in D_i\\
&D_i(s_i\beta) = r_iD''_i(\beta), ~\beta \in im\, v_i \subset F_i
\end{align*}
\end{proof}

\section{Local Construction}\label{LocalConstruction}

\subsection{Denotations and Assumptions}\label{DenAss} Throughout Section \hyp{LocalConstruction} we let $k$ denote a field and $X$ a compact subvariety of a real analytic manifold with singular set $\Sigma$.
Assume $X$ admits pc tubular data and is oriented of dimension $n$ (e.g.\ $X$ is complex and equidimensional). 
Let $T$ denote a pc tubular neighborhood of $\Sigma$.

If $\bar{p}$ is a perversity function, for $i \in \Z$ we define: 
\begin{align*}
Z^{\bar{p}}_i &= im\left[H_i(\partial T;k) \rightarrow IH^{\bar{p}}_i(T;k)\right],~ Z_{\bar{p}}^i = (Z^{\bar{p}}_i)^*\\
Y^{\bar{p}}_i &= coker\left[IH^{\bar{p}}_i(T;k) \rightarrow IH^{\bar{p}}_i(T,\partial T;k)\right],~ Y_{\bar{p}}^i = (Y^{\bar{p}}_i)^*.
\end{align*}
We also write $Z^{\bar{p}}_\bullet$ and $Y^{\bar{p}}_\bullet$ (resp. $Z_{\bar{p}}^\bullet$ and $Y_{\bar{p}}^\bullet$) if we'd like to interpret these collections of vector spaces as chain (resp. cochain) complexes with zero differential.

\subsection{Duality and the Image of the Boundary}\label{DualityImage}
Let $\bar{p}$ denote a perversity function.
To begin, we'd like to overcome the obstruction discussed in Section \hyp{Motivation}.
It will be essential to work with the \textit{image} of $H_*(\partial T;k) \rightarrow IH^{\bar{p}}_*(T;k)$ as opposed to $IH^{\bar{p}}_*(T;k)$ itself, the latter being more in line with the original approach.
We've already denoted this collection of vector spaces by $Z^{\bar{p}}_*$.
The first step in this transition is to understand the ``Lefschetz dual'' object to $Z^{\bar{p}}_*$, in the sense of Theorem \hyp{tubneighdual}. 
 
\begin{lem}\label{identifications}
For all $i \in \Z$ there are natural identifications:
\begin{align*}
Z_{\bar{p}}^i &= coim\left[IH_{\bar{p}}^i(T;k) \rightarrow H^i(\partial T;k)\right]\\
Y_{\bar{p}}^i &= ker\left[IH_{\bar{p}}^i(T,\partial T;k) \rightarrow IH_{\bar{p}}^i(T;k)\right] 
\end{align*}
\end{lem}
\begin{proof}
If we apply universal coefficients (see \cite{Fr} Theorem 7.1.4 for the intersection cohomology version of universal coefficients) and the second identification of Lemma \hyp{dualmap} to the map $H_i(\partial T;k) \rightarrow IH^{\bar{p}}_i(T;k)$ then we obtain the identification:
\begin{align*}Z_{\bar{p}}^i = coim \left[IH_{\bar{p}}^i(T;k) \rightarrow H^i(\partial T;k)\right].\end{align*}
If we next apply universal coefficients and the first identification of Lemma \hyp{dualmap} to the map $IH^{\bar{p}}_i(T;k) \rightarrow IH^{\bar{p}}_i(T,\partial T;k)$ then we obtain the identification:
\begin{align*}Y_{\bar{p}}^i = ker \left[IH_{\bar{p}}^i(T,\partial T;k) \rightarrow IH_{\bar{p}}^i(T;k)\right].\end{align*} 
\end{proof}

\begin{lem}\label{lefschetzdual}
Suppose $(\bar{p},\bar{q})$ are complementary perversity functions.
There exists an exact triangle of chain complexes with zero differential:
\begin{align*}H_\bullet(\partial T; k) \rightarrow Z^{\bar{p}}_\bullet \rightarrow Y^{\bar{p}}_\bullet \xrightarrow{-1}\end{align*}
and an exact triangle of cochain complexes with zero differential:
\begin{align*}Y_{\bar{q}}^\bullet \rightarrow Z_{\bar{q}}^\bullet \rightarrow H^\bullet(\partial T;k) \xrightarrow{+1}.\end{align*}
Moreover there is a natural duality isomorphism between their (shifted) long exact sequences:
$$\begin{tikzcd}
\cdots \ar[r] & H_i(\partial T; k) \ar[d,"D", "\simeq" {anchor=south, rotate=90, inner sep=1mm}] \ar[r] & Z^{\bar{p}}_i \ar[d,"D","\simeq" {anchor=south, rotate=90, inner sep=1mm}] \ar[r] & Y^{\bar{p}}_i \ar[d,"D", "\simeq" {anchor=south, rotate=90, inner sep=1mm}] \ar[r] & \cdots\\
\cdots \ar[r] & H^{n-1-i}(\partial T;k) \ar[r] & Y_{\bar{q}}^{n-i} \ar[r] & Z_{\bar{q}}^{n-i} \ar[r] & \cdots
\end{tikzcd}$$
\end{lem}
\begin{proof}
Existence of the first exact triangle is a direct consequence of Lemma \hyp{indexacttri} and Theorem \hyp{tubneighdual}.
Existence of the second exact triangle is a consequence of Lemma \hyp{coindexacttri}, Theorem \hyp{tubneighdual}, and Lemma \hyp{identifications}.

For the isomorphism of long exact sequences, we first recall the diagram from Theorem \hyp{tubneighdual}:
$$\begin{tikzcd}
\cdots \ar[r] & H_i(\partial T;R) \ar[d, "D", "\simeq" {anchor=south, rotate=90, inner sep=1mm}] \ar[r] & IH^{\bar{p}}_i(T;R) \ar[d, "D", "\simeq" {anchor=south, rotate=90, inner sep=1mm}] \ar[r] & IH^{\bar{p}}_i(T,\partial T;R) \ar[d, "D", "\simeq" {anchor=south, rotate=90, inner sep=1mm}] \ar[r] & \cdots\\
\cdots \ar[r] & H^{n-i-1}(\partial T;R) \ar[r] & IH_{\bar{q}}^{n-i}(T,\partial T;R) \ar[r] & IH_{\bar{q}}^{n-i}(T;R) \ar[r] & \cdots  
\end{tikzcd}.$$
Because \textit{this} diagram is an isomorphism of long exact sequences, appealing to the definition of $Z^{\bar{p}}_i$ we have:
\begin{align*}
D(Z^{\bar{p}}_i) &= im \left[H^{n-1-i}(\partial T;k) \rightarrow IH_{\bar{q}}^{n-i}(T,\partial T;k)\right]\\
&= ker \left[IH_{\bar{q}}^{n-i}(T,\partial T;k) \rightarrow IH_{\bar{q}}^{n-i}(T;k)\right] = Y_{\bar{q}}^{n-i}
\end{align*} 
where in the last step we have used Lemma \hyp{identifications}.
This provides us the middle isomorphism of the desired diagram.

To construct the rightmost isomorphism of the desired diagram we first observe that:
\begin{align*}
D\left(im \left[IH^{\bar{p}}_i(T;k) \rightarrow IH^{\bar{p}}_i(T,\partial T;k)\right]\right) &= im \left[IH_{\bar{q}}^{n-i}(T,\partial T;k) \rightarrow IH_{\bar{q}}^{n-i}(T;k)\right]\\
&= ker \left[IH_{\bar{q}}^{n-i}(T;k) \rightarrow H^{n-i}(\partial T;k)\right]. 
\end{align*}
Therefore $D$ induces an isomorphism between $Y^{\bar{p}}_i$ and $Z_{\bar{q}}^{n-i}$ after we make the identification of Lemma \hyp{identifications}.

That these isomorphisms fit into a \textit{commutative} diagram follows from the fact that all our maps are induced from the \textit{already existing} commutative diagram of long exact sequences from which we have been drawing. 
\end{proof}

\subsection{Local Intersection Approximation}
Let $\bar{p}$ denote a perversity function.
The results of Section \hyp{DualityImage} will be the source of various desired properties for our intersection space construction. 
Banagl constructs his intersection spaces by first selecting a ``local approximation'', local in the sense that it takes as input only the tubular neighborhood of the singular set.
We will do the same, but will also loosen some constraints by allowing an approximation which is merely ``algebraic'', not necessarily topological.    
\\

\noindent A \bi{$\bar{p}$ algebraic intersection approximation for $T$ with coefficients in $k$} is a pair $(A_\bullet,f_\bullet)$ where $A_\bullet$ is a chain complex of $k$-vector spaces and:
\begin{align*} f_\bullet: A_\bullet \rightarrow C_\bullet(\partial T;k)\end{align*}  
is a chain map such that the composition:
\begin{align*}H_\bullet(A_\bullet) \xrightarrow{f_*} H_\bullet(\partial T;k) \rightarrow Z^{\bar{p}}_\bullet\end{align*}
is an isomorphism.
A \bi{$\bar{p}$ topological intersection approximation for $T$ with coefficients in $k$} is a pair $(A,f)$ where $A$ is a topological space and $f: A \rightarrow \partial T$ is a continuous map such that $(C_\bullet(A;k), f_{\#})$ is a local $\bar{p}$ algebraic intersection approximation.

Observe that the class of such approximations does not depend on the particular choice of tubular neighborhood (see Lemmas \hyp{strathomeo} and \hyp{invariance} and their proofs). 
These are extensions of Banagl's ``approximations'' as indicated by the following examples.

\begin{ex}
Suppose $\Sigma = \{x\}$ is a single point.
Then the tubular neighborhood of $\Sigma$ is conic: $T = cL$ with cone point $x$ where $L$ is called the link of $x$.
So $\partial T = L$ and the cone formula (see \cite{Fr} Theorem 4.2.1) implies:
\begin{align*}IH^{\bar{p}}_\bullet(T;k) = H_\bullet^{<n-1-\bar{p}(n)}(L;k)\end{align*}
where $H_i^{<n-1-\bar{p}(n)}(L;k)$ agrees with $H_i(L;k)$ for $i < n-1-\bar{p}(n)$ and vanishes otherwise.
The map:
\begin{align*}H_\bullet(L;k) \rightarrow H_\bullet^{<n-1-\bar{p}(n)}(L;k)\end{align*}
is surjective so that $Z^{\bar{p}}_\bullet = H_\bullet^{<n-1-\bar{p}(n)}(L;k)$.
Consider a Moore approximation (defined in \cite{Ba1}):
\begin{align*}f: L_{<n-1-\bar{p}(n)} \rightarrow L.\end{align*}
By its defining properties, the composition:
\begin{align*}H_\bullet(L_{<n-1-\bar{p}(n)};\Q) \xrightarrow{f_*} H_\bullet(L;\Q) \rightarrow H_\bullet^{<n-1-\bar{p}(n)}(L;\Q)\end{align*}
is an isomorphism.
Therefore $(L_{<n-1-\bar{p}(n)},f)$ is a $\bar{p}$ topological intersection approximation for $T$ with coefficients in $\Q$. $\sslash$
\end{ex}

\begin{ex}\label{BC1}
Suppose $X$ has a Whitney stratification consisting of exactly two strata $\{X-\Sigma, \Sigma\}$ where $\Sigma$ has codimension $c$.
In particular, this means that $\Sigma$ is smooth and connected. 
Suppose also that $T$ is homeomorphic to the mapping cylinder of a fiber bundle projection $\partial T \rightarrow \Sigma$ (e.g.\ $T$ is a tubular neighborhood in the Thom-Mather sense).
Suppose there exists a fiberwise truncation (defined in \cite{BaC}):
\begin{align*}f: ft_{<c-1-\bar{p}(c)}\partial T \rightarrow \partial T.\end{align*}
By \cite{BaC} Proposition 6.5, the composition:
\begin{align*}H_\bullet(ft_{<c-1-\bar{p}(c)}\partial T;\Q) \xrightarrow{f_*} H_\bullet(\partial T;\Q) \rightarrow IH^{\bar{p}}_\bullet(T;\Q)\end{align*}
is an isomorphism.
Therefore $(ft_{<c-1-\bar{p}(c)}\partial T,f)$ is a $\bar{p}$ topological intersection approximation for $T$ with coefficients in $\Q$. $\sslash$  
\end{ex}

\begin{prop}[Existence]\label{exist1}
A $\bar{p}$ algebraic intersection approximation $(A_\bullet, f_\bullet)$ for $T$ with coefficients in $k$ always exists.
\end{prop}
\begin{proof}
Pick a section $s$ of the composition:
\begin{align*} ker\left[\partial_\bullet: C_\bullet(\partial T;k) \rightarrow C_{\bullet-1}(\partial T;k)\right] \rightarrow H_\bullet(\partial T;k) \rightarrow Z^{\bar{p}}_\bullet\end{align*}
where the first map is the quotient map from cycles to homology classes and the second is the obvious surjection.
The composition:
\begin{align*} Z^{\bar{p}}_\bullet \xrightarrow{s} ker\, \partial_\bullet \hookrightarrow C_\bullet(\partial T;k)\end{align*}
is an algebraic intersection approximation.
\end{proof}

We next examine the ``dual'' object to an intersection approximation.

\begin{lem}\label{cone}
Suppose $(A_\bullet, f_\bullet)$ is a $\bar{p}$ algebraic intersection approximation for $T$ with coefficients in $k$.
Then the composition (where the first map is the boundary map of Lemma \hyp{lefschetzdual}):
\begin{align*}Y^{\bar{p}}_{\bullet+1} \rightarrow H_\bullet(\partial T;k) \rightarrow H_\bullet(cf_\bullet)\end{align*}
is an isomorphism, where $c f_\bullet$ denotes the algebraic cone on $f_\bullet$.
\end{lem}
\begin{proof}
We will describe three exact triangles, then use the octahedral axiom for triangulated categories (see \cite{W}) to construct a fourth that will imply the theorem.

By definition of ``algebraic cone'', we have an exact triangle:
\begin{align*}A_\bullet \xrightarrow{f_*} C_\bullet(\partial T;k) \rightarrow cf_\bullet \xrightarrow{-1}.\end{align*}
A consequence of the definition of algebraic intersection approximation is that the map $f_*: H_\bullet(A_\bullet) \rightarrow H_\bullet(\partial T;k)$ is injective.
Applying this to the long exact sequence from the aforementioned exact triangle shows that the map $H_\bullet(\partial T;k) \rightarrow H_\bullet(cf_\bullet)$ is surjective
and the boundary map $H_\bullet(cf_\bullet) \rightarrow H_{\bullet-1}(A_\bullet)$ is the zero map.
We next translate this exact triangle to obtain another (with the same long exact sequence in homology):
\begin{align*}C_\bullet(\partial T;k) \rightarrow cf_\bullet \rightarrow A_{\bullet - 1} \xrightarrow{-1}.\end{align*} 
We apply Lemma \hyp{indexacttri} and our observations about the maps from this exact triangle to find an exact triangle:
\begin{align*}H_\bullet(\partial T;k) \rightarrow H_\bullet(cf_\bullet) \rightarrow H_{\bullet-1}(A_\bullet) \xrightarrow{-1}\end{align*}
and translate it back to obtain another exact triangle:
\begin{align*}H_\bullet(A_\bullet) \rightarrow H_\bullet(\partial T;k) \rightarrow H_\bullet(cf_\bullet) \xrightarrow{-1}.\end{align*}

By definition of an algebraic intersection approximation, there is an isomorphism of chain complexes (with zero differential) $H_\bullet(A_\bullet) \xrightarrow{\simeq} Z^{\bar{p}}_\bullet$.
Consequently there is an exact triangle:
\begin{align*}H_\bullet(A_\bullet) \rightarrow Z^{\bar{p}}_\bullet \rightarrow 0 \xrightarrow{-1}.\end{align*} 

Lastly, by Lemma \hyp{lefschetzdual} there is an exact triangle:
\begin{align*}H_\bullet(\partial T;k) \rightarrow Z^{\bar{p}}_\bullet \rightarrow Y^{\bar{p}}_\bullet \xrightarrow{-1}.\end{align*}

Altogether we have exact triangles:
\begin{align*}
&H_\bullet(A_\bullet) \rightarrow H_\bullet(\partial T;k) \rightarrow H_\bullet(cf_\bullet) \xrightarrow{-1}\\
&H_\bullet(\partial T;k) \rightarrow Z^{\bar{p}}_\bullet \rightarrow Y^{\bar{p}}_\bullet \xrightarrow{-1}\\
&H_\bullet(A_\bullet) \rightarrow Z^{\bar{p}}_\bullet \rightarrow 0 \xrightarrow{-1}.
\end{align*}
where the first map of the third exact triangle is the composition $H_\bullet(A_\bullet) \rightarrow H_\bullet(\partial T;k) \rightarrow Z^{\bar{p}}_\bullet$.
This is the setting in which the octahedral axiom is applicable.
The resulting exact triangle is:
\begin{align*}H_\bullet(cf_\bullet) \rightarrow 0 \rightarrow Y^{\bar{p}}_\bullet \xrightarrow{-1}\end{align*}
and the boundary map (which must be an isomorphism) is the composition:
\begin{align*}Y^{\bar{p}}_{\bullet+1} \rightarrow H_\bullet(\partial T;k) \rightarrow H_\bullet(cf_\bullet).\end{align*}
\end{proof}

\subsection{Local Duality Isomorphism}\label{LocalDualityIsomorphism}
Let $(\bar{p}, \bar{q})$ denote complementary perversities.
If our intersection spaces are to have a global duality, a local duality must first be understood.
Suppose $(A^{\bar{p}}_\bullet,f^{\bar{p}}_\bullet)$ and $(A^{\bar{q}}_\bullet,f^{\bar{q}}_\bullet)$ are $\bar{p}$ and $\bar{q}$ algebraic intersection approximations for $T$ with coefficients in $k$.
We consider the diagram:
$$\begin{tikzcd}
Z_{\bar{q}}^{n-r-1} \ar[r] & H^{n-r-1}(\partial T;k) \ar[r] & H^{n-r-1}(A^{\bar{q}}_\bullet)\\
Y^{\bar{p}}_{r+1} \ar[u, "D"] \ar[r] & H_r(\partial T; k) \ar[u,"D" ] \ar[r] & H_r(cf^{\bar{p}}_\bullet).
\end{tikzcd}$$ 
The upper composition is by definition an isomorphism and the lower composition is by Lemma \hyp{cone} also an isomorphism.
So there exists a unique \bi{local duality isomorphism}:
\begin{align*} D: H_r(cf^{\bar{p}}_\bullet) \rightarrow H^{n-r-1}(A^{\bar{q}}_\bullet)\end{align*}
that makes the \textit{outer} box commute.
This describes a \bi{local intersection pairing}:
\begin{align*}
(-,-) : H^{\bar{p}}_r(cf^{\bar{p}}_\bullet) \times H_{n-r-1}(A^{\bar{q}}_\bullet) \rightarrow k,~ (\alpha,\beta) = D(\alpha)(\beta). 
\end{align*} 

We say the \bi{$r$th local duality obstruction for $(A^{\bar{p}}_\bullet, f^{\bar{p}}_\bullet)$, $(A^{\bar{q}}_\bullet, f^{\bar{q}}_\bullet)$ vanishes} if and only if the \textit{entire} diagram:
$$\begin{tikzcd}
Z_{\bar{q}}^{n-r-1} \ar[r] & H^{n-r-1}(\partial T;k) \ar[r] & H^{n-r-1}(A^{\bar{q}}_\bullet)\\
Y^{\bar{p}}_{r+1} \ar[u, "D"] \ar[r] & H_r(\partial T; k) \ar[u,"D"] \ar[r] & \ar[u, "D"] H_r(cf^{\bar{p}}_\bullet)
\end{tikzcd}$$ 
commutes (this is not necessarily true, because the right box need not commute).
The following theorem captures the physical notion that the local duality obstructions will vanish if $im(f_*)$ contains no pairs of ``stably intersecting'' cycles.    

\begin{prop}\label{intpairing}
The $r$th local duality obstruction for $(A^{\bar{p}}_\bullet, f^{\bar{p}}_\bullet)$, $(A^{\bar{q}}_\bullet, f^{\bar{q}}_\bullet)$ vanishes if and only if given any:
\begin{align*}\alpha \in im\,f_*^{\bar{p}} \subset H_\bullet(\partial T;k), \beta \in im\,f_*^{\bar{q}} \subset H_\bullet(\partial T;k) \mbox{ with $|\alpha| = r,~ |\beta| = n-r-1$} \end{align*}
we have the following vanishing of the intersection pairing on $\partial T$:
\begin{align*}(\alpha,\beta) = 0.\end{align*}
This vanishing occurs for example if $\alpha$ and $\beta$ are representable by disjoint cycles. 
\end{prop}
\begin{proof}
Fix $r \in \Z$ and consider the commutative diagram whose maps have been named:
$$\begin{tikzcd}
Z_{\bar{q}}^{n-r-1} \ar[r, "l"] & H^{n-r-1}(\partial T;k) \ar[r, "f_{\bar{q}}^*"] & H^{n-r-1}(A^{\bar{q}}_\bullet)\\
Y^{\bar{p}}_{r+1} \ar[u, "D"] \ar[r, "u"] & H_r(\partial T; k) \ar[u,"D"] \ar[r, "v"] & H_r(cf^{\bar{p}}_\bullet).
\end{tikzcd}$$ 
The $r$th local duality obstruction vanishes iff:
\begin{align*} f_{\bar{q}}^*D = (f_{\bar{q}}^*l)D(vu)^{-1}v\end{align*}
By commutativity we have equivalences:
\begin{align*}\left[f_{\bar{q}}^*D = (f_{\bar{q}}^*l)D(vu)^{-1}v\right] &\iff \left[f^*D = f_{\bar{q}}^*Du(vu)^{-1}v\right]\\ &\iff \left[\forall \alpha \in H_r(\partial T;k),~ f_{\bar{q}}^*D(\alpha - u(vu)^{-1}v\alpha) = 0\right].\end{align*}
Next observe that:
\begin{align*}
\{\alpha - u(vu)^{-1}v\alpha \mid \alpha \in H_r(\partial T;k)\} = ker\, v = im \left[f^{\bar{p}}_*: H_r(A^{\bar{p}}_\bullet) \rightarrow H_r(\partial T;k)\right];
\end{align*}
The second equality is a consequence of a long exact sequence.
The ``$\subset$'' part of the first equality can be directly verified.
For the ``$\supset$'' part of this equality, simply observe that if $\alpha \in ker\, v$, then:
\begin{align*}\alpha = \alpha - u(vu)^{-1}v\alpha.\end{align*}
Put together we have that the $r$th local duality obstruction vanishes iff:
\begin{align*}
D\left(im\left[f^{\bar{p}}_*: H_r(A^{\bar{p}}_\bullet) \rightarrow H_r(\partial T;k)\right] \right) \subset ker\, \left[f_{\bar{q}}^*: H^{n-r-1}(\partial T;k) \rightarrow H^{n-r-1}(A^{\bar{q}}_\bullet)\right].
\end{align*}
This holds iff for all $\alpha \in im\, f^{\bar{p}}_*$ with $|\alpha| = r$ and $\beta = f^{\bar{q}}_*(\gamma) \in im\, f^{\bar{q}}_*$ with $|\beta| = n-r-1$ we have:
$$0 = f_{\bar{q}}^*D(\alpha)(\gamma) = D(\alpha)(\beta) = (\alpha,\beta)$$
where we have used the fact that the duality isomorphism induces the intersection pairing on homology.
\end{proof}

\begin{ex}
Suppose $\Sigma = \{p\}$ consists of a single point with link $L$.
Then:
\begin{align*}
IH^{\bar{p}}_\bullet(T;k) = H_\bullet^{<n-1-\bar{p}(n)}(L;k), \ IH^{\bar{q}}_\bullet(T;k) = H_\bullet^{<n-1-\bar{q}(n)}(L;k)
\end{align*}
Therefore, given any algebraic intersection approximations $(A^{\bar{p}}_\bullet, f^{\bar{p}}_\bullet)$, $(A^{\bar{q}}_\bullet, f^{\bar{q}}_\bullet)$ for $T$ with coefficients in $k$, the subsets $im\,f_*^{\bar{p}}$ and $im\,f^{\bar{q}}_*$ of $H_\bullet(\partial T;k)$ contain classes of degree strictly less than $n-1-\bar{p}(n)$ and $n-1-\bar{q}(n)$ respectively.
By definition of complementary perversities we have:
\begin{align*}
(n-1-\bar{p}(n))-1+(n-1-\bar{q}(n))-1 =	n-2.
\end{align*}
Since $\partial T$ is an $(n-1)$-dimensional manifold, no two of these classes can pair to a nonzero field element.
So the local duality obstructions always vanish. $\sslash$
\end{ex}

\begin{ex}
In this example we show that our local duality obstructions all vanish if and only if those of Banagl-Chriestenson \cite{BaC} all vanish (when our local intersection approximations are fiberwise truncations).
We assume $\Q$-coefficients.
Suppose $X$ has a Whitney stratification consisting of exactly two strata $\{X-\Sigma, \Sigma\}$ where $\Sigma$ has codimension $c$.
Suppose also that $T$ is homeomorphic to the mapping cylinder of a fiber bundle projection $\partial T \rightarrow \Sigma$, and that there exist fiberwise truncations:
\begin{align*}&f^{\bar{p}}: ft_{<c-1-\bar{p}(c)}\partial T \rightarrow \partial T.\\
&f^{\bar{q}}: ft_{<c-1-\bar{q}(c)}\partial T \rightarrow \partial T.
\end{align*}
By Example \hyp{BC1} these maps are shown to be topological local intersection approximations and:
\begin{align*}
Z_\bullet^{\bar{p}} = IH_\bullet^{\bar{p}}(T),\ Z_\bullet^{\bar{q}} = IH_\bullet^{\bar{q}}(T)	
\end{align*}
from which we also conclude:
\begin{align*}
Y_\bullet^{\bar{p}} = IH_\bullet^{\bar{p}}(T,\partial T),\ Y_\bullet^{\bar{q}} = IH_\bullet^{\bar{q}}(T,\partial T)	.
\end{align*}
There are also isomorphisms:
\begin{align*}
H_\bullet(cf_\#^{\bar{p}}) \cong \widetilde{H}_\bullet(cf^{\bar{p}}),\ H_\bullet(cf_\#^{\bar{p}}) \cong \widetilde{H}_\bullet(cf^{\bar{p}})
\end{align*}
where $cf^{\bar{p}}$ and $cf^{\bar{q}}$ are topological mapping cones. 
Therefore by definition our $r$th local duality obstruction vanishes if and only if the entire diagram:
$$\begin{tikzcd}
IH_{\bar{q}}^{n-r-1}(T) \ar[r] & H^{n-r-1}(\partial T) \ar[r] & H^{n-r-1}(ft_{<c-1-\bar{q}(c)}\partial T)\\
IH^{\bar{p}}_{r+1}(T,\partial T) \ar[u, "D"] \ar[r] & H_r(\partial T) \ar[u,"D"] \ar[r] & \ar[u, "D"] \widetilde{H}_r(cf^{\bar{p}})
\end{tikzcd}$$
commutes.
Compare this to the diagram appearing in \cite[Proposition 6.10]{BaC}: up to labelling and the direction of duality isomorphisms, it is the same, since their $Q_{\geq c-1-\bar{p}(c)}$ is homotopy equivalent to our $cf^{\bar{p}}$ by \cite[Equation (6.4)]{BaC}.
Banagl-Chriestenson prove that their local duality obstructions vanish if and only if the above diagram commutes for all $r$.
In other words: the local duality obstructions of Banagl-Chriestenson all vanish if and only if our local duality obstructions all vanish (where both are associated to a fixed complementary pair of fiberwise truncations).    
$\sslash$	
\end{ex}

\textit{Remark.} If $X$ is a Witt space (see \cite{Fr} Definition 9.1.2 and Proposition 9.1.8) then so is $T$. 
In this case, by definition, the approximations for $T$ for the lower $\bar{m}$ and upper $\bar{n}$ middle perversities would be indistinguishable.
So for $X$ Witt, an $\bar{m}$ approximation $(A_\bullet, f_\bullet)$ is an $\bar{n}$ approximation - we call this a \bi{Witt approximation for $T$ with coefficients in $k$} - and we can talk about vanishing of duality obstructions for $(A_\bullet, f_\bullet)$ alone.
For simplicity, we state the following theorem for Witt approximations.

\begin{thm}[Existence]\label{exist2}
Suppose $X$ is a Witt space of even dimension $n = 2m$.
Then there exists a Witt algebraic intersection approximation $(A_\bullet,f_\bullet)$ for $T$ with coefficients in $k$ for which all the local duality obstructions vanish.
\end{thm}
\begin{proof}
We suppress perversity superscripts and subscripts since they yield isomorphic objects below.
Pick a Witt algebraic intersection approximation $(A_\bullet, f_\bullet)$ with \textit{zero differential} as in Proposition \hyp{exist1}. 
This allows us to assume $A_\bullet \subset H_\bullet(\partial T;k)$; since $A_\bullet = H_\bullet(A_\bullet)$ and the map $H_\bullet(A_\bullet) \rightarrow H_\bullet(\partial T;k)$ is injective by definition.

We will replace all the $A_r$ for $r \geq m$ and leave unchanged all the $A_r$ for $r < m$.
We do this as follows.
\textit{Fix} $r < m$ and set $s = 2m-1-r \geq m$.
\textit{Replace} our given $A_s$ with the subspace of $H_s(\partial T;k)$ on which $D(A_r) \subset H^s(\partial T;k)$ vanishes.
We pick a map:
\begin{align*} f_s: A_s \rightarrow C_s(\partial T;k) \end{align*}
by selecting a section of the quotient map \{$s$-cycles of $\partial T$\} $\to$ $H_s(\partial T;k)$, and then using the composition:
$$
\begin{tikzcd}[column sep = small] A_s \ar[r,hook] & H_s(\partial T;k) \ar[r, dashed, bend left = 10] & \ar[l, bend left = 10] \{\mbox{s-cycles}\} \ar[r,hook] & C_s(\partial T;k).\end{tikzcd}
$$

We must now check two things.
First, that the composition:
\begin{align*} A_s \hookrightarrow H_s(\partial T;k) \rightarrow Z_s\end{align*}
is an isomorphism.
Second, that the $r$th local duality obstruction vanishes (observe that Proposition \hyp{intpairing} implies we only have to check vanishing for $r < n$, because this leads to vanishing for all $r$).
Because the differentials of our old and new $A_\bullet$ are zero, there is nothing else to worry about.
 
By Lemma \hyp{lefschetzdual} we have a commutative diagram of short exact sequences (where we have arbitrarily assigned names to some maps and considered injections as inclusions):
$$
\begin{tikzcd}
0 \ar[r] & Z^s \ar[r, hook] & H^s(\partial T;k) \ar[r, "u"] & Y^{s+1} \ar[r] & 0\\
0 \ar[r] & Y_{r+1} \ar[u,"D", "\simeq" {anchor=south, rotate=-90, inner sep=1mm}] \ar[r, hook] & H_r(\partial T;k) \ar[u, "D", "\simeq" {anchor=south, rotate=-90, inner sep=1mm}] \ar[r, "v"] & Z_r \ar[u, "D", "\simeq" {anchor=south, rotate=-90, inner sep=1mm}] \ar[r] & 0. 
\end{tikzcd}
$$ 
By definition and from the diagram we have $dim\, A_r = dim\, Z_r = \dim\, Y^{s+1}$.
By construction, the codimension of $A_s$ in $H^s(\partial T;k)$ is equal to the dimension of $A_r$.
Combining this with the diagram of short exact sequences we have:
\begin{align*}
 dim\, A_s = dim\, H^s(\partial T;k) - dim\, Z_r &= dim\, H^s(\partial T;k) - dim\, Y^{s+1}\\
&= dim\, Z^s = dim\, Z_s.
\end{align*}  
This is a start, for we have shown that the dimensions of $A_s$ and $Z_s$ coincide. 
We now need only show that $A_s$ does has trivial intersection with the kernel of $H_s(\partial T;k) \rightarrow Z_s$ to conclude that $A_s$ is suitable for an algebraic intersection approximation.
Suppose towards a contradiction that it has \textit{nontrivial} intersection with this kernel.
From the short exact sequence:
\begin{align*} 0 \rightarrow Y_{s+1} \hookrightarrow H_s(\partial T;k) \rightarrow Z_s \rightarrow 0 \end{align*}
if follows that $A_s \cap Y_{s+1}$ is nontrivial.
Pick a function $g \in Y^{s+1}$ which does not vanish on $A_s \cap Y_{s+1}$.
Our diagram identifies $Y^{s+1}$ with $Z_r$, and $A_r \subset H_r(\partial T;k)$ maps isomorphically onto $Z_r$ under $v$, so there exists $\alpha \in A_r$ with:
\begin{align*} g = Dv(\alpha) = uD(\alpha) \end{align*} 
The map $u$ is none other than the restriction to $Y_{s+1}$.
So the fact that $g = uD(\alpha)$ does not vanish on $A_s \cap Y_{s+1}$ implies $D(\alpha)$ does not vanish on $A_s$.
This contradicts the definition of $A_s$.
Hence, in fact $A_s \cap Y_{s+1} = \langle 0 \rangle$ and the composition:
$$A_s \hookrightarrow H_s(\partial T;k) \rightarrow Z_s$$
is an isomorphism.

Next we verify vanishing of the $r$th local duality obstruction.
We know that $D(A_r)$ vanishes on $A_s$.
Therefore for all $\alpha \in A_r$ and $\beta \in A_s$ we have;
$$D(\alpha)(\beta) = (\alpha,\beta) = 0.$$
Now apply Proposition \hyp{intpairing}.
\end{proof}

\section{Global Construction}\label{GlobalConstruction}

\subsection{Denotations and Assumptions}\label{DenAssGlobal} 
Throughout Section \hyp{GlobalConstruction} we let $k$ denote a field and $X$ a compact subvariety of a real analytic manifold with singular set $\Sigma$.
Assume $X$ admits pc tubular data and is oriented of dimension $n$ (e.g.\ $X$ is complex and equidimensional). 
Let $T$ denote a pc tubular neighborhood of $\Sigma$.
Let $(\bar{p},\bar{q})$ denote complementary perversity functions.
Let $(A_\bullet^{\bar{p}}, f_\bullet^{\bar{p}})$ and $(A_\bullet^{\bar{q}}, f_\bullet^{\bar{q}})$, respectively $(A^{\bar{p}}, f^{\bar{p}})$ and $(A^{\bar{q}}, f^{\bar{q}})$, denote algebraic, respectively topological, intersection approximations for $T$ with coefficients in $k$.

\subsection{Intersection Space} We are now in a position to define a global space extending earlier definitions of intersection space.
The \bi{algebraic intersection space $I_{f^{\bar{p}}_\bullet}X$ associated to $(A^{\bar{p}}_\bullet, f^{\bar{p}}_\bullet)$} is the algebraic cone on the composition:
\begin{align*}
A^{\bar{p}}_\bullet \xrightarrow{f^{\bar{p}}_\bullet} C_\bullet(\partial T;k) \xrightarrow{incl_\bullet} C_\bullet(X-T^\circ;k).
\end{align*}
The \bi{topological intersection space $I_{f^{\bar{p}}}X$ associated to $(A^{\bar{p}},f^{\bar{p}})$} is the topological cone on the composition:
\begin{align*}
A^{\bar{p}} \xrightarrow{f^{\bar{p}}} \partial T \xrightarrow{incl} X - T^\circ.
\end{align*} 
We achieve a global duality assuming the local duality obstructions vanish.

\begin{thm}\label{algdual}
Assume the local duality obstructions vanish for $(A^{\bar{p}}_\bullet, f^{\bar{p}}_\bullet)$, $(A^{\bar{q}}_\bullet, f^{\bar{q}}_\bullet)$.
Then there exist non-canonical duality isomorphisms:
\begin{align*}
D: H_r(I_{f^{\bar{p}}_\bullet}X) \xrightarrow{\simeq} H^{n-r}(I_{f^{\bar{q}}_\bullet}X).
\end{align*}
\end{thm}
\begin{proof}
We temporarily omit the perversity superscripts and subscripts, as the following statements about distinguished triangles hold for both.
We have a set of three distinguished triangles:
\begin{align*}
&A_\bullet \xrightarrow{f_\bullet} C_\bullet(\partial T; k) \rightarrow cf_\bullet \xrightarrow{-1}\\
&C_\bullet(\partial T; k) \xrightarrow{incl_\bullet} C_\bullet(X - T^\circ; k) \rightarrow C_\bullet(X-T^\circ,\partial T;k) \xrightarrow{-1}\\
&A_\bullet \xrightarrow{incl_\bullet \circ f_\bullet} C_\bullet(X-T^\circ;k) \rightarrow I_{f_\bullet}X \xrightarrow{-1}.
\end{align*}
The octahedral axiom implies the existence of a third distinguished triangle:
$$cf_\bullet \rightarrow I_{f_\bullet}X \rightarrow C_\bullet(X-T^\circ,\partial T;k) \xrightarrow{-1}.$$
The octahedral axiom moreover relates the maps in these four distinguished triangles; namely we have the following (every map below is a map from one of these distinguished triangles, and a shift by ``$-1$'' in a subscript indicates we are considering a boundary map): 
\begin{align}
&\left[C_{\bullet}(X-T^\circ, \partial T;k) \rightarrow cf_{\bullet-1}\right] = \left[C_{\bullet}(X-T^\circ,\partial T;k) \rightarrow C_{\bullet-1}(\partial T;k) \rightarrow cf_{\bullet-1}\right]\\
&\left[cf_{\bullet} \rightarrow A_{\bullet-1}\right] = \left[cf_{\bullet} \rightarrow I_{f_{\bullet}}X  \rightarrow A_{\bullet-1}\right]\\
&\left[C_\bullet(X-T^\circ;k) \rightarrow C_\bullet(X-T^\circ, \partial T;k)\right] = \left[C_\bullet(X-T^\circ;k) \rightarrow I_{f_\bullet}X \rightarrow C_\bullet(X-T^\circ, \partial T;k)\right]\\
&\left[C_\bullet(\partial T;k) \rightarrow cf_\bullet \rightarrow I_{f_\bullet} X\right] = \left[C_\bullet(\partial T;k) \rightarrow C_\bullet(X-T^\circ;k) \rightarrow I_{f_\bullet}X\right]\\
&\left[I_{f_\bullet}X \rightarrow C_\bullet(X-T^\circ, \partial T;k) \rightarrow C_{\bullet-1}(\partial T;k)\right] = \left[I_{f_\bullet}X \rightarrow A_{\bullet-1} \rightarrow C_{\bullet-1}(\partial T;k)\right]
\end{align}
We will only use the first of these in this proof, but the rest will be important later.

We now reintroduce perversity subscripts and superscripts.
Consider two long exact sequences obtained from the aforementioned distinguished triangles:
$$\begin{tikzcd}[column sep = small]
\cdots \ar[r] & H^{n-r-1}(A^{\bar{q}}_\bullet) \ar[r] & H^{n-r}(I_{f^{\bar{q}}_\bullet}X) \ar[r] & H^{n-r}(X-T^\circ;k) \ar[r] & H^{n-r}(A^{\bar{q}}_\bullet) \ar[r] & \cdots\\
\cdots \ar[r] & H_{r}(cf^{\bar{p}}_\bullet) \ar[u, "D"] \ar[r] & H_{r}(I_{f^{\bar{p}}_\bullet}X) \ar[r] & H_{r}(X-T^\circ, \partial T; k) \ar[u, "D"] \ar[r] & H_{r-1}(cf^{\bar{p}}_\bullet) \ar[u, "D"] \ar[r] & \cdots. 
\end{tikzcd}$$
If we can prove that this diagram is commutative, then we can use Lemma \hyp{dualconstruction} to construct (non-canonical) duality isomorphisms.
The left (bigger) rectangle above commutes by exactness.
We next use vanishing of duality obstructions to show that the right square also commutes.
By our observation (1) about the boundary map of the lower long exact sequence, the square of interest can be decomposed:
$$\begin{tikzcd}
H^{n-r}(X - T^\circ;k) \ar[r] & H^{n-r}(\partial T;k) \ar[r] & H^{n-r}(A^{\bar{q}}_\bullet)\\
H_{r}(X-T^\circ,\partial T;k) \ar[u,"D"] \ar[r] & H_{r-1}(\partial T;k) \ar[u,"D"] \ar[r] & H_{r-1}(cf^{\bar{p}}_\bullet) \ar[u,"D"] 
\end{tikzcd}$$
The leftmost box in this decomposed diagram always commutes, and the rightmost box commutes owing to the vanishing of the $(r-1)$th local duality obstruction.

We have successfully verified the hypotheses of Lemma \hyp{dualconstruction}.
\end{proof}

There is an analogous statement for the topological intersection space.

\begin{cor}
Assume the local duality obstructions for $(A^{\bar{p}},f^{\bar{p}})$, $(A^{\bar{q}},f^{\bar{q}})$  vanish.
Then there exist non-canonical duality isomorphisms:
\begin{align*} D: \tilde{H}_r(I_{f^{\bar{p}}}X; k) \xrightarrow{\simeq} \tilde{H}^{n-r}(I_{f^{\bar{q}}}X; k). \end{align*} 
\end{cor}
\begin{proof}
This is a consequence of the arguments from Theorem \hyp{algdual}, since there for either perversity (omit the superscripts) there is an exact triangle:
\begin{align*}
C_\bullet(A;k) \xrightarrow{incl_\bullet \circ f_{\#}} C_\bullet(X - T^\circ;k) \rightarrow \tilde{C}_\bullet(I_fX) \xrightarrow{-1}.
\end{align*}
associated to a topological mapping cone.
\end{proof}

\textit{Remark.} While the cochain complex $C^\bullet(I_fX)$ of a \textit{topological intersection space} is naturally a differential graded $k$-algebra under cup product, the dual complex $(I_{f_\bullet}X)^*$ of an \textit{algebraic} intersection space does not seem to have a natural multiplicative structure.
Therefore finding topological, as opposed to just algebraic,  intersection spaces will prove to be an interesting task.    

\begin{ex}
We show that, when the local intersection approximation is a fiberwise truncation, our topological intersection space coincides with the Banagl-Chriestenson intersection space.
With assumptions as in Example \hyp{BC1} we have a fiberwise truncation $f: ft_{<c-1-\bar{p}(c)} \partial T \rightarrow \partial T$.
As in Example \hyp{BC1}, this fiberwise truncation constitutes a topological intersection approximation for $T$.
The associated topological intersection space is the cone on the composition:
\begin{align*}
ft_{<c-1-\bar{p}(c)} \partial T \rightarrow \partial T \rightarrow X - T^\circ.	
\end{align*}
This is precisely \cite[Definition 9.2]{BaC}, the definition of the Banagl-Chriestenson intersection space. $\sslash$
\end{ex}

\textit{Remark.} There is not an obvious general sheaf interpretation of algebraic intersection space cohomology.
This is because the local intersection approximation takes as input the not entirely local map $C_\bullet(\partial T; k) \rightarrow IH_\bullet(T;k)$. 
This in contrast to the \textit{AF intersection space pairs} of \cite{AF}, but we will show in the following subsection that the AF intersection space is in general distinct from our algebraic intersection space: in an example, we will show that homologies of the two do not even coincide.

\section{A Worked out Example}\label{Example}
\subsection{Denotations and Assumptions} 
In this section, we will deal only with spaces with even-dimensional strata, so without further comment we use middle-perversity intersection homology. 
Let $X \subset \C P^2$ denote an irreducible degree three nodal hypersurface with exactly one singular point $p$.
Let $B$ denote a closed tubular neighborhood of $p$ in $X$ whose boundary is denoted by $L$.
Let $M$ denote $X-B^\circ$. Observe that:
\begin{enumerate}[-]
\item $X$ is topologically a pinched torus.
\item $B \cong cL$.
\item $L \cong S^1 \sqcup S^1$.
\item $M \cong S^1 \times D^1$.
\end{enumerate}
Let $\overline{\calX} \subset \C P^3$ denote the projective cone on $X$.
Let $\infty$ denote $(0:0:0:1) \in \C P^3$.
The vector bundle $\C \hookrightarrow \mathbb{C}P^3-\{\infty\} \rightarrow \C P^2$ restricts to a vector bundle $\C \hookrightarrow \overline{\calX}-\{\infty\} \rightarrow X$ which we denote by $(\calX,\pi)$.
We also let:
\begin{enumerate}[-]
\item $\rho$ denote the restriction of vector bundle $\calX$ over $p$.
\item $\calB$ denote the restriction of vector bundle $\calX$ over $B$.
\item $\calL$ denote the restriction of vector bundle $\calX$ over $L$.
\item $\calM$ denote the restriction of vector bundle $\calX$ over $M$.
\item $\overline{\rho}$ denote the line in $\C P^3$ connecting $p$ with $\infty$, i.e.\ the closure of $\rho$ in $\overline{\calX}$, also the singular set of $\overline{\calX}$. 
\item $S(-)$ denote the sphere bundle associated to a vector bundle $(-)$.
\item $D(-)$ denote the disk bundle associated to a vector bundle $(-)$.
\end{enumerate}

\subsection{Setting Up the Example}
We will explicitly construct a topological intersection space for the projective cone $\overline{\calX}$ on $X \subset \C P^2$.
Moreover we will show that the corresponding local duality obstructions vanish.     
This example is of interest, because it is depth two with pseudomanifold stratification $\overline{\calX} \supset \overline{\rho} \supset \{\infty\}$, so the topological methods of \cite{BaC} do not apply.
We will also use this example to distinguish our construction from the construction of \cite{AF}.  

To rigorously carry out this construction we need to analyze $\overline{\calX}$ in detail. 
Topologically it is the Thom space of the vector bundle $(\calX, \pi)$ as described for example in \cite[Page 18]{Bo}, therefore is the homotopy pushout of the following diagram involving disk and circle bundles:
\begin{align*}
\overline{\calX} = hp\left(D\calX \leftarrow S\calX \rightarrow cS\calX\right).
\end{align*}
where $cS\calX$ is the cone on $S\calX$.
The singular set $\overline{\rho}$ is the homotopy pushout:
\begin{align*}
\overline{\rho} = hp\left(D\rho \leftarrow S\rho \rightarrow cS\rho\right)	
\end{align*}
We then define:
\begin{align*}
T = hp\left(D\calB \leftarrow S\calB \rightarrow cS\calX\right).	
\end{align*}

\begin{clm}\label{hpdescriptions}
$T$ is a closed tubular neighborhood of $\overline{\rho}$ in $\overline{\calX}$ in the sense of the remark following Lemma \hyp{invariance}.
Moreover, the nonsingular boundary $\partial T$ is the homotopy pushout:
\begin{align*}
	\partial T = hp\left(D\calL \leftarrow S\calL \rightarrow S\calM\right)
\end{align*}
and $X-T^\circ$ is the homotopy pushout:
\begin{align*}
	X-T^\circ = hp\left(D\calM \leftarrow S\calM \rightarrow S\calM\right)
\end{align*}
\end{clm} 
\begin{proof}
Consider the inclusions of diagrams:
$$\begin{tikzcd}[column sep = small]
D\calX & \ar[l] S\calX \ar[r] & cS\calX\\
D\calB \ar[u] & \ar[l] S\calB \ar[u] \ar[r] & cS\calX \ar[u]\\
D\rho \ar[u] & \ar[l] S\rho \ar[u] \ar[r] & cS\rho \ar[u].
\end{tikzcd}$$
Each vertical inclusion from the lower half of the diagram is the inclusion of a deformation retract, so that the inclusion of homotopy pushouts $\overline{\rho} \rightarrow T$ is also the inclusion of a deformation retract.
This provides our family $T(\epsilon)$ as in the remark following Lemma \hyp{invariance}.
Therefore, once we check in the following paragraph that $\partial T$ is a submanifold of $\overline{\calX}$, we have our tubular neighborhood $T$. 

By inspection, the boundary of $\partial T$ in $\overline{\calX}$ is the homotopy pushout:
$$hp\left(D\calL \leftarrow S\calL \rightarrow S\calM\right).$$
Because $D\calL$ and $S\calM$ are manifolds with boundary $S\calL$, it follows that the homotopy pushout $\partial T$ is in fact a closed manifold. 
By another inspection the complement $\overline{\calX} - T^\circ$ is the homotopy pushout:
$$hp\left(D\calM \leftarrow S\calM \rightarrow S\calM\right)$$
which again is a manifold with boundary $\partial T$.
In particular, $\partial T$ is a submanifold of $\overline{\calX}$.

\end{proof}

We will need to analyze various homologies and intersection homologies related to the tube in order to construct an intersection space.
First we find:

\begin{clm}
\begin{align*}
IH_*(S\calX) = \left\{\begin{array}{ll}
\Z & \mbox{if } * = 0,3\\
\Z_3 & \mbox{if } * = 1\\
0 & \mbox{otherwise.}
\end{array}\right.
\end{align*}
\end{clm}
\begin{proof}
Since $X$ is a pinched torus, it has normalization $\nu: S^2 \rightarrow X$ where $S^2$ is a two-sphere.
Let $S\mathcal{Z}$ and $\hat{\nu}$ be such that the below is a pair of pullback diagrams:
$$\begin{tikzcd}
S\mathcal{Z} \ar[d] \ar[r, "\hat{\nu}"] & S\calX \ar[d] \ar[r] & S^5 \ar[d] \\
S^2 \ar[r, "\nu"] & X \ar[r] & \C P^2.
\end{tikzcd}$$
Then $\hat{\nu}$ is a normalization and $S\mathcal{Z}$ is a principal circle bundle over $S^2$. 
Each principal circle bundle corresponds to an element of $H^2(S^2;\Z)$. 
Let us determine to which element $S\mathcal{Z}$ corresponds. 

Since $S^5 \rightarrow \C P^2$ is the pullback of the universal circle-bundle $S^\infty \rightarrow \C P^\infty$ under the inclusion, by composing pullbacks it follows that $S\mathcal{Z} \rightarrow S^2$ is the pullback of the universal bundle under:
\begin{align*}
S^2 \xrightarrow{\nu} X \hookrightarrow \C P^2 \hookrightarrow \C P^{\infty}. 	
\end{align*}
We analyze this map on second cohomology.
Both $\nu$ and $\C P^2 \hookrightarrow \C P^{\infty}$ induce isomorphisms on second cohomology.
Because $X$ is degree $3$ the map $\Z \cong H^2(\C P^2;\Z) \rightarrow H^2(X;\Z) \cong \Z$ is multiplication by $\pm 3$.
Therefore the composition $H^2(\C P^{\infty}; \Z) \rightarrow H^2(S^2; \Z)$ is multiplication by $\pm 3$.
Hence $S\mathcal{Z}$ is the unique principal circle bundle corresponding to $\pm 3 \in H^2(S^2;\Z)$.
A standard argument then shows that:
\begin{align*}
H_*(S\mathcal{Z}) = \left\{\begin{array}{ll}
\Z & \mbox{if } * = 0,3\\
\Z_3 & \mbox{if } * = 1\\
0 & \mbox{otherwise.}
\end{array}\right.
\end{align*}
But $\hat{\nu}: S\mathcal{Z} \rightarrow S\calX$ is a normalization, so as detailed in \cite[I.1.6]{Bo} induces an isomorphism $H_*(S\mathcal{Z}) \cong IH_*(S\calX)$.
\end{proof}

Next we work on the rational intersection homology of tubular neighborhood $T$.

\begin{clm}
\begin{align*}
IH_*(T; \Q) = \left\{\begin{array}{ll}
\Q & \mbox{if } * = 0\\
\Q^2 & \mbox{if } * = 2\\
0 & \mbox{otherwise.}
\end{array}\right.
\end{align*}
\end{clm}
\begin{proof}
Given the description of $T$ as the homotopy pushout of:
\begin{align*}
D\calB \leftarrow S\calB \rightarrow cS\calX	
\end{align*}
and the fact that $D\calB$ and $S\calB$ are circle bundles over contractible $B$, we obtain a long exact sequence:
$$\cdots \rightarrow IH_i(B \times S^1;\Q) \rightarrow IH_i(B \times D^2;\Q) \oplus IH_i(cS\calX;\Q) \rightarrow IH_i(T;\Q) \rightarrow \cdots$$
which, using the cone formula and K\"{u}nneth for intersection homology (see \cite[Theorem 4.2.1, Corollary 6.4.10]{Fr}), becomes:
$$\cdots \rightarrow \left[H_*^{<1}(L;\Q) \otimes H_*(S^1;\Q)\right]_i \rightarrow \left[H_*^{<1}(L;\Q) \otimes H_*(D^2;\Q)\right]_i \oplus IH_i^{< 2}(S\calX;\Q) \rightarrow IH_i(T;\Q) \rightarrow \cdots.$$ 
Recall that $L \cong S^1 \sqcup S^1$.
For $i = 0$ the first map is given by:
\begin{align*}
H_0(L;\Q) \rightarrow H_0(L;\Q) \oplus IH_0(S\calX;\Q).
\end{align*}
which is obviously injective.
Then for $i = 1$ we have exact:
\begin{align*}
H_0(L;\Q) \otimes H_1(S^1;\Q) \rightarrow IH_1(S\calX;\Q) \rightarrow IH_1(T;\Q) \rightarrow 0.
\end{align*}
But $IH_1(S\calX;\Q) = 0 \implies IH_1(T; \Q) = 0$.
For $i = 2$ we have:
\begin{align*}
0 \rightarrow IH_2(T;\Q) \rightarrow H_0(L;\Q) \otimes H_1(S^1;\Q) \rightarrow 0
\end{align*}
which implies $IH_2(T;\Q) \cong \Q^2$. 
The remaining homology $\Q$-vector spaces are trivially computed owing to the vanishing of many terms in the long exact sequence. 
\end{proof}

Next let's provide exact descriptions for $X-T^\circ$ and $\partial T$.

\begin{clm}\label{hompush}
$X-T^\circ \cong S^1 \times D^3$ and $\partial T \cong S^1 \times S^2$. 
\end{clm}
\begin{proof}
By Claim \hyp{hpdescriptions} the space $X-T^\circ$ is the homotopy pushout:
\begin{align*}
hp\left(D\calM \leftarrow S\calM \rightarrow S\calM\right)	
\end{align*}
which (since $S\calM$ has a collar neighborhood in $D\calM$) is homeomorphic to $D\calM$.
But $H^2(M;\Z) = 0$ so the complex vector bundle $\calM$ is trivial.
Hence:
\begin{align*}
D\calM \cong M \times D^2 \cong S^1 \times D^1 \times D^2 \cong S^1 \times D^3.	
\end{align*}
\end{proof}

\subsection{The Intersection Space}

We rely on the results of the preceding subsection to construct a topological intersection approximation.
Let $S^2 \xrightarrow{f} S^1 \times S^2 \cong \partial T$ be the inclusion of a sphere such that $f$ induces an isomorphism on second homology.

\begin{clm}\label{intapproxex}
The pair $(S^2, f)$ is a topological intersection approximation for $T$ with coefficients in $\Q$.   
\end{clm}
\begin{proof}
We first must understand the map $H_\bullet(\partial T;\Q) \rightarrow IH_\bullet(T;\Q)$ and its image $Z_\bullet$.
For dimensional reasons, the description of the map is only unclear in degree two.
In this case we have exact:
$$IH_3(T,\partial T;\Q) \rightarrow H_2(\partial T;\Q) \rightarrow IH_2(T;\Q).$$
Duality shows that:
$$IH_3(T,\partial T;\Q) \cong IH_1(T;\Q)^* = 0$$
and consequently that the induced map $H_2(\partial T;\Q) \rightarrow IH_2(T;\Q)$ is an injection.
We explicitly specify:
$$Z_* = \left\{\begin{array}{ll}
im \left[H_0(\partial T;\Q) \hookrightarrow IH_0(T;\Q)\right] & \mbox{if } * = 0\\
im \left[H_2(\partial T;\Q) \hookrightarrow IH_2(T;\Q)\right] & \mbox{if } * = 2\\
0 & \mbox{otherwise.}
\end{array}\right.$$ 
The map $H_*(S^2;\Q) \xrightarrow{f_*} H_*(\partial T;\Q)$ is an isomorphism for $* = 0,2$ and $H_*(S^2;\Q)$ vanishes otherwise.
So the composition:
$$H_\bullet(S^2;\Q) \xrightarrow{f_*} H_\bullet(\partial T;\Q) \rightarrow Z_\bullet$$
is an isomorphism.   
\end{proof}

We have a topological intersection approximation for the tube $T$, so are granted a topological intersection space $I_f\overline{\calX}$.
It is obtained by coning off an embedded $S^2$ in the boundary $\partial T \cong S^1 \times S^2$ of $\overline{\calX} - T^\circ$.
The long exact sequence associated to the inclusion $S^2 \hookrightarrow S^1 \times D^3 \cong \overline{X}-T^\circ$ gives:
\begin{align*}
H_*(I_f\overline{\calX};\Q) = \left\{\begin{array}{ll}
\Q & \mbox{if}\ * = 0,1,3\\
0 & \mbox{otherwise.}	
\end{array}\right.
\end{align*}
Alternatively, check that $I_f\overline{\calX} \simeq S^1 \vee S^3$.
The dual Betti numbers of $I_f\overline{\calX}$ seem to match up.
In fact this is because:

\begin{clm}
The local duality obstructions vanish for the intersection approximation $S^2 \xrightarrow{f} \partial T$.
Therefore, the intersection space $I_f \overline{\calX}$ satisfies duality.
\end{clm}
\begin{proof}
As usual we set:
\begin{align*}
Z_\bullet &= im\, H_\bullet(\partial T;\Q) \rightarrow IH_\bullet(T;\Q)\\
Y_\bullet &= coker\, H_\bullet(T;\Q) \rightarrow IH_\bullet(T,\partial T;\Q). 
\end{align*}
Consider the diagram:
$$\begin{tikzcd}
Z_{3-r}^* \ar[r] & H_{3-r}(\partial T; \Q)^* \ar[r, "f^*"] & H_{3-r}(S^2;\Q)^*\\
Y_{r+1} \ar[u, "D"] \ar[r] & H_r(\partial T;\Q) \ar[u, "D"] \ar[r] & H_r(cf_\#) \ar[u, "D"] 
\end{tikzcd}$$
which a priori need not commute.
The left box always commutes, and the outer box commutes by construction.
The local duality obstruction vanishes if and only if the right box also commutes.

Commutativity is obvious when $r \neq 1,3$ because the upper-rightmost term vanishes.
When $r = 1$ or $r = 3$ the map $f^*$ is an isomorphism (for $r  = 1$ see the proof of Claim \hyp{intapproxex}), from which it can derived that all maps are isomorphisms, in which case the box again commutes (owing to the commutativity of the left box and the outer box).
\end{proof}

\textit{Remark.} With some effort, this example can be extended to the projective cone on \textit{any} irreducible hypersurface in $\C P^2$ with isolated singularity.
In this general case, the topological local intersection approximation will be composed of a wedge of spheres and circles.
Again, the local duality obstructions will vanish.

\medskip

\subsection{Comparison with the AF intersection space.}
We can compare with the method introduced in \cite{AF}, and will show that their \textit{AF intersection space pair} is distinct from our notion of algebraic intersection space even on the level of homology.
Since the strata $\rho$ and $\{\infty\}$ are contractible, \cite[Theorem 3.30]{AF} implies that their construction applies.
We avoid excruciating detail, choosing only to outline the construction of this AF intersection space pair $\left(I\overline{\calX}_{AF}, \overline{\rho}_{AF}\right)$. 

\cite{AF} requires a decomposition of the tubular neighborhood $T$, which we provide in this paragraph.
Keeping in mind the homotopy pushout descriptions:
\begin{align*}
\overline{\rho} = hp\left(D\rho \leftarrow S\rho \rightarrow cS\rho\right),\ T = hp\left(D\calB \leftarrow S\calB \rightarrow cS\calX\right)	
\end{align*}
with cone point $\infty$,
we define:
\begin{align*}
\overline{\rho}_1 &= hp\left(D\rho \leftarrow S\rho \rightarrow S\rho\right),\ T_1 = hp\left(D\calB \leftarrow S\calB \rightarrow S\calB\right),\ E_1 = hp\left(D\calL \leftarrow S\calL \rightarrow S\calL\right)\\
\overline{\rho}_0 &= cS\rho,\ T_0 = cS\calX,\ E_0 = S\calX.
\end{align*}
Observe that:
\begin{enumerate}[-]
\item $T_1 \cap \overline{\rho} = \overline{\rho}_1$, $T_0 \cap \overline{\rho} = \overline{\rho}_0$, and $\overline{\rho}_0 \cap \overline{\rho}_1 = S\rho$.
\item $E_1 = \partial T_1 \cap (X-T^\circ)$ and $E_0 = \partial T_0$.
\item $E_1$ fibers trivially over $\overline{\rho}_1$ with fiber $L$.
\item $T_1 = cyl\left(E_1 \rightarrow \overline{\rho}_1\right)$, the mapping cylinder of the bundle projection.
\item the pair $(E_0, S\rho)$ fibers trivally over $\{\infty\}$ with fiber $(S\calX, S\rho)$.
\item $(T_0, \overline{\rho}_0) \cong cyl\left((E_0,S\rho) \rightarrow \{\infty\}\right)$, the mapping cylinder pair of the pair of bundle projections. Since $\{\infty\}$ is a point set, this is actually a cone pair.
\item $\partial T_0 \cap \partial T_1 = S\calL = E_1|_{S\rho}$, the restricted fiber bundle over $S\rho$.
\item  $\partial T_0 \cap T_1 = S\calB = cyl\left(E_1|_{S\rho} \rightarrow S\rho \right)$, the mapping cylinder of the bundle projection.
\item  $\partial T_0 = S\calM \cup_{E_1|_{S\rho}} cyl\left(E_1|_{S\rho}\rightarrow S\rho\right)$
\end{enumerate}
Keep these observations in mind when considering the construction detailed in the following paragraph.

The essence of the \cite{AF} construction for this example (up to homotopy, not word-for-word) is the following: 
\begin{enumerate}[(i)]

\item Select a fiberwise truncation $ft_{<1}E_1 \rightarrow E_1$ of the trivial bundle $E_1 \rightarrow \overline{\rho}_1$.
Define $T_1^{AF} = cyl\left(ft_{<1}E_1 \rightarrow \overline{\rho}_1\right)$, the mapping cylinder of the bundle projection.

\item Define a Step 1 AF intersection space:
\begin{align*}
I\overline{\calX}_{AF, 1} = (X-T^\circ) \cup_{ft_{<1}E_1}  T_1^{AF} 	
\end{align*}
by gluing $T_1^{AF}$ to $X-T^\circ$ via:
\begin{align*}
T_1^{AF} \hookleftarrow ft_{<1}E_1 \rightarrow E_1 \hookrightarrow X-T^\circ	
\end{align*}
Effectively, we have deleted $T$ from $X$, then replaced $T_1$ with $T_1^{AF}$.

\item Define pair:
\begin{align*}
	\partial T_0^{AF} = S\calM \cup_{ft_{<1}E_1|_{S\rho}} cyl\left(ft_{<1}E_1|_{S\rho} \rightarrow S\rho\right) \subset I\overline{\calX}_{AF,1}.
\end{align*}
and interpret $\left(\partial T_0^{AF}, S\rho\right)$ as a pair of fiber bundles over the point set $\{\infty\}$.

\item Select a fiberwise truncation of pairs $\left(\partial T_0^{AF}, S\rho\right)_{<2} \rightarrow (\partial T_0^{AF}, S\rho)$ of the pair of fiber bundles $\left(\partial T_0^{AF}, S\rho\right) \rightarrow \{\infty\}$.
Define $(T_0^{AF}, \overline{\rho}_0^{AF}) = cyl\left((\partial T_0^{AF}, S\rho)_{<2} \rightarrow \{\infty\}\right)$, the mapping cylinder pair of the pair of bundle projections.

\item Define the AF intersection space pair:
\begin{align*}
	\left(I\overline{\calX}_{AF}, \overline{\rho}_{AF}\right) = (I\overline{\calX}_{AF,1}, \overline{\rho}_1) \cup_{\left(\partial T_0^{AF}, S\rho\right)_{<2}} \left(T_0^{AF}, \overline{\rho}_0^{AF}\right)  
\end{align*}
by gluing $(T_0^{AF}, \overline{\rho}_0^{AF})$ to $(I\overline{\calX}_{AF,1}, \overline{\rho}_1)$ via:
\begin{align*}
\left(T_0^{AF},\overline{\rho}_0^{AF}\right) \hookleftarrow \left(\partial T_0^{AF}, S\rho\right)_{<2} \rightarrow \left(\partial T_0^{AF}, S\rho\right) \hookrightarrow \left(I\overline{\calX}_{AF,1}, \overline{\rho}_1\right). 	
\end{align*}
Effectively, we have replaced $(T_0, \overline{\rho}_0)$ with $\left(T_0^{AF}, \overline{\rho}_0^{AF}\right)$.
\end{enumerate}

The rational homology of the pair $\left(I\overline{\calX}_{AF}, \overline{\rho}_{AF}\right)$ has the potential to satisfy duality, and is what we will compare the rational homology of our algebraic intersection spaces against. 
Having outlined the construction, let's select fiberwise truncations and determine an explicit AF intersection space.

We include the following claims without proof, as they can be verified in a straightforward manner.

\begin{clm}
Let $L_{<1} = \{*\} \sqcup \{*\} \hookrightarrow L$ be the inclusion of two points into the two disjoint circles that make up $L$.
Let $\calL_{<1}$ denote the restriction of bundle $\calL \rightarrow L$ above subspace $L_{<1}$.
The space $\calL_{<1}$ like $\calL$ can also be interpreted as a trivial bundle over $\rho$, but with fiber $L_{<1}$.
Then:
\begin{align*}
ft_{<1}E_1 = hp\left(D\calL_{<1} \leftarrow S\calL_{<1} \rightarrow S\calL_{<1} \right)	\hookrightarrow E_1 
\end{align*}
is a fiberwise truncation of bundles over $\overline{\rho}_1$.
It is an inclusion.
\end{clm}

\begin{clm}
Define $B^{AF} = cyl\left(L_{<1} \rightarrow p\right)$.
It is a subset of $cyl(L \rightarrow p) = B$.
Let $\calB^{AF}$ denote the restriction of $\calB \rightarrow B$ above subspace $B^{AF}$.
Then $T_1^{AF}$ from (i) of the AF construction is:
\begin{align*}
T_1^{AF} = hp\left(D\calB^{AF} \leftarrow S\calB^{AF} \rightarrow S\calB^{AF}\right).	
\end{align*}
It is a subset of $T_1$.
\end{clm}

\begin{clm}\label{step1}
Define $IX^{AF} = M \cup B^{AF}$. 
It is a subset of $X$ that is homotopy equivalent to a wedge of two circles.
Let $\calI\calX^{AF}$ denote the restriction of $\calX \rightarrow X$ above subspace $IX^{AF}$.
It is a trivial bundle because $IX^{AF}$ has vanishing second cohomology.
Then the Step 1 AF intersection space $I\overline{\calX}_{AF,1}$ is:
\begin{align*}
I\overline{\calX}_{AF,1} = hp\left(D\calI\calX^{AF} \leftarrow S\calI\calX^{AF} \rightarrow S\calI\calX^{AF}\right)	
\end{align*}
and $\partial T_0^{AF}$ from (iii) of the AF construction is:
\begin{align*}
\partial T_0^{AF} = S\calI\calX^{AF}.
\end{align*}
It is a subset of $\partial T_0$.
\end{clm}

\begin{clm}\label{fiberwisepair}
The pair $\left(\partial T_0^{AF}, S\rho\right) = \left(S\calI\calX^{AF}, S\rho\right)$, interpreted as a pair of bundles over $\{\infty\}$, has fiberwise truncation:
\begin{align*}
\left(S\calI\calX^{AF}, S\rho\right)_{<2} = \left(IX^{AF},p\right) \hookrightarrow \left(S\calI\calX^{AF}, S\rho\right)	
\end{align*}
where the inclusion is any section of the trivial bundle pair $\left(S\calI\calX^{AF}, S\rho\right) \rightarrow \left(IX^{AF},p\right)$.	
\end{clm}

Because all our truncation are inclusions, the associated AF intersection space $I\overline{\calX}_{AF}$ is a subset of $\overline{\calX}$. 
We describe the pair $\left(I\overline{\calX}_{AF}, \overline{\rho}_{AF}\right)$:

\begin{clm}
Both $I\overline{\calX}_{AF}$ and $\overline{\rho}_{AF}$ are contractible.
Therefore $H_*\left(I\overline{\calX}_{AF}, \overline{\rho}_{AF}\right)$ vanishes identically.
\end{clm}
\begin{proof}
By Claims \hyp{step1} and \hyp{fiberwisepair} and the AF construction, one verifies that the AF intersection space pair is described by the following mapping cones:
\begin{align*}
I\overline{\calX}_{AF} &\cong c\left(IX^{AF} \hookrightarrow D\calI\calX^{AF}\right)\\
\overline{\rho}_{AF} &\cong c\left(p \hookrightarrow D\rho\right) 
\end{align*}
where $(IX^{AF}, p) \hookrightarrow (D\calI\calX^{AF}, D\rho)$ is the inclusion of a section of the trivial bundle pair $\left(S\calI\calX^{AF}, S\rho\right) \rightarrow \left(IX^{AF}, p\right)$. 
But $IX^{AF} \subset D\calI\calX^{AF}$ and $p \subset D\rho$ are deformation retracts. 
Therefore $I\overline{\calX}_{AF}$ and $\overline{\rho}_{AF}$ are contractible.
\end{proof}

We have shown that $H_*\left(I\overline{\calX}_{AF}, \overline{\rho}_{AF}\right)$ vanishes identically. 
On the other hand, suppose we are given \textit{any} algebraic intersection approximation $(A_{\bullet}, f_{\bullet})$ for $T$ with coefficients in $\Q$, and associated algebraic intersection space $I_{f_\bullet}X$.
Then we have the following exact sequence:
\begin{align*}
H_1(A_\bullet) \rightarrow H_1(\partial T;\Q) \rightarrow H_1(I_{f_\bullet}X)	
\end{align*}
But $H_1(A_\bullet) \cong Z_1 = 0$ (see Proof of Claim \hyp{intapproxex}) and $H_1(\partial T;\Q) \cong \Q$ together imply $H_1(I_{f_\bullet}X) \neq 0$.
In other words, our notion of intersection space is distinct from the AF notion.
It seems difficult to compare them in general.

\section{Intersection Space Pairing}\label{IntersectionSpacePairing}

\subsection{Denotations and Assumptions}
Throughout Section \hyp{IntersectionSpacePairing} we let $k$ denote a field and $X$ a compact subvariety of a real analytic manifold with singular set $\Sigma$.
Assume $X$ admits pc tubular data and is oriented of \textit{even} dimension $2n$ (e.g.\ $X$ is complex and equidimensional). 
Let $T$ denote a pc tubular neighborhood of $\Sigma$.
Assume $X$ is a \textit{Witt space} and $(A_\bullet, f_\bullet)$ is a \textit{Witt} algebraic approximation for $T$ with coefficients in $k$ \textit{for which the local duality obstructions vanish}.
Recall that an approximation for a Witt space is said to be Witt if it is either a lower $\bar{m}$ or upper $\bar{n}$ middle perversity approximation, and that distinguishing between the two is unnecessary as the constructed objects are naturally isomorphic. 
We thus omit any perversity subscripts and superscripts (assuming them to be either $\bar{m}$ or $\bar{n}$, distinction unnecessary). 

We also use this subsection as a grand collection of \textit{names and properties of maps}.
We give names to the following natural maps, all of which sit inside exact sequences (see preceding sections to understand these sequences):
\begin{align*}
&H_\bullet(cf_\bullet) \xrightarrow{u_\bullet} H_\bullet(I_{f_\bullet}X) \xrightarrow{v_\bullet} H_\bullet(X-T^\circ, \partial T;k)\\
&H_\bullet(X-T^\circ;k) \xrightarrow{h_\bullet} H_\bullet(I_{f_\bullet}X) \xrightarrow {g_\bullet} H_{\bullet - 1}(A_\bullet)\\
&H_\bullet(\partial T;k) \xrightarrow{\iota_\bullet} H_\bullet(X-T^\circ;k) \xrightarrow{j_\bullet} H_\bullet(X-T^\circ,\partial T;k) \xrightarrow{\delta_\bullet} H_{\bullet-1}(\partial T;k)\\
&H_\bullet(\partial T;k) \xrightarrow{\ell_\bullet} H_\bullet(cf_\bullet) \xrightarrow{0} H_{\bullet - 1}(A_\bullet) \xrightarrow{f_{\bullet-1}} H_{\bullet - 1}(\partial T;k)
\end{align*}
where maps that sit in the same row are sequential in a long exact sequence.
Next we gather the relationships between these maps (all of which can be found in the proof of Theorem \hyp{algdual}):
\begin{align*}
&j_\bullet = h_\bullet v_\bullet\\
&0 = g_\bullet u_\bullet\\
&h_\bullet \iota_\bullet = u_\bullet \ell_\bullet\\
&f_{\bullet-1}g_\bullet = \delta_\bullet v_\bullet\\
&\mbox{$\ell_\bullet$ is surjective and $f_\bullet$ is injective.}
\end{align*}
We consider it allowable to use these properties \textit{without comment}.
The dual of a map, $u_\bullet$ for example, will be denoted by $u^\bullet$.
Lastly we name the duality isomorphisms:
\begin{align*}
&D_{f_\bullet}: H_\bullet(cf_\bullet) \xrightarrow{\simeq} H^{2n-1-\bullet}(A_\bullet)\\
&D_\partial: H_\bullet(\partial T;k) \xrightarrow{\simeq} H^{2n-1-\bullet}(\partial T;k)\\
&D_L: H_\bullet(X-T^\circ, \partial T;k) \xrightarrow{\simeq} H^{2n-\bullet}(X-T^\circ;k)\\ 
&D'_L: H_\bullet(X-T^\circ;k) \xrightarrow{\simeq} H^{2n-\bullet}(X-T^\circ, \partial T;k) 
\end{align*}
where ``$L$'' indicates Lefschetz duality.
For $\alpha \in H_\bullet(X-T^\circ;k)$ and $\beta \in H_{2n-\bullet}(X-T^\circ, \partial T;k)$ we have:
\begin{align*}
D_L'(\alpha)(\beta) = (-1)^{|\alpha||\beta|} D_L(\beta)(\alpha)
\end{align*}
since these duality isomorphisms (or more specifically their inverses) can be understood in terms of cup products, which are anti-commutative.
The duality isomorphisms are related to each other as follows (again these properties are allowable to use \textit{without comment}):
\begin{align*}
&D_{f_\bullet} \ell_\bullet = f^{2n-1-\bullet}D_\partial\\
&D_L' \iota_\bullet = \delta^{2n-\bullet} D_\partial\\
&D_L j_\bullet = j^{2n-\bullet} D_L'.
\end{align*} 
where the first is a direct consequence of the local duality obstructions vanishing, and the second two follow from commutativity of the duality isomorphism diagram relating the long exact sequence of the pair $(X-T^\circ, \partial T;k)$ in homology to the long exact sequence of the pair in cohomology. 

\subsection{Families of Sections}
We would like our duality isomorphisms on the intersection space to have some geometric significance, and to give us a meaningful intersection space pairing.
In this subsection, we describe how duality isomorphisms are selected.

Lemma \hyp{dualconstruction} gives us insight into the particular nature of a duality isomorphism $D_{IX}: H_\bullet(I_{f_\bullet}X) \rightarrow H^{2n-\bullet}(I_{f_\bullet}X)$.
Consider the commutative diagram of exact sequences from the previous section:
$$\begin{tikzcd}
\cdots \ar[r] & H^{2n-i-1}(A_\bullet) \ar[r, "g^{2n-i}"] & H^{2n-i}(I_{f_\bullet}X) \ar[r, "h^{2n-i}"] & H^{2n-i}(X-T^\circ;k) \ar[r] & \cdots\\
\cdots \ar[r] & H_{i}(cf_\bullet) \ar[u, "D_{f_\bullet}"] \ar[r, "u_i"] & H_{i}(I_{f_\bullet}X) \ar[r, "v_i"] & H_{i}(X-T^\circ, \partial T; k) \ar[u, "D_L"] \ar[r] & \cdots. 
\end{tikzcd}$$
where $g^{2n-i} = (g_{2n-i})^*$ and $h^{2n-i} = (h_{2n-i})^*$ are dual maps to the maps on homology.
By Lemma \hyp{dualconstruction}, the intersection space duality isomorphism  is constructed by selecting \bi{families of sections}:
\begin{align*}
&r^\bullet: im\, h^\bullet \rightarrow H^\bullet(I_{f_\bullet}X)\\
&s_\bullet: im\, v_\bullet \rightarrow H_\bullet(I_{f_\bullet}X)
\end{align*}
where we will utilize the identification $im\, h^{\bullet} = (coim\, h_\bullet)^*$ asserted by Lemma \hyp{dualmap} to write $r^\bullet = (r_\bullet)^*$ where $r_\bullet: H_\bullet(I_{f_\bullet X}) \rightarrow coim\, h_\bullet$ is a retraction of $h_\bullet: coim\, h_\bullet \hookrightarrow H_\bullet(I_{f_\bullet}X)$.
As in the proof of Lemma \hyp{dualconstruction} (and with indices shifted for the family of sections on cohomology) we thus have duality isomorphism:
\begin{align*}D_{IX} = D^{(r^\bullet, s_\bullet)}: H_\bullet(I_{f_\bullet}X) \xrightarrow{\simeq} H^{2n-\bullet}(I_{f_\bullet}X)\end{align*}
which is entirely described by:
\begin{align*}
&\alpha \in H_i(cf_\bullet) \implies D_{IX}(u_i\alpha) = g^{2n-i}D_{f_\bullet}\alpha\\
&\beta \in im\, v_i \implies D_{IX}(s_i\beta) = r^{2n-i}D_L\beta. 
\end{align*}
This will allow us to describe the \bi{intersection space pairing associated to $(r^\bullet, s_\bullet)$}, which is defined by:
\begin{align*}
(-,-)_{IX} = (-,-)^{(r^\bullet, s_\bullet)} : H_i(I_{f_\bullet}X) \times H_{2n-i}(I_{f_\bullet}X) \rightarrow k,~ (\alpha,\beta)_{IX}  = D_{IX}(\alpha)(\beta).
\end{align*}

\begin{lem}\label{intspacepair}
Suppose $(r^\bullet, s_\bullet)$ is a family of sections and $i \in \Z$. 
Consider the decompositions:
\begin{align*} 
&H_i(I_{f_\bullet} X) = im\, u_i \oplus im\, s_i\\
&H_{2n-i}(I_{f_\bullet} X) = im\, h_{2n-i} \oplus ker\, r_{2n-i}.  
\end{align*}
Under the intersection space pairing for $(r^\bullet, s_\bullet)$ we have:
\begin{align*}
\left(im\, u_i, im\, h_{2n-i}\right)_{IX} = 0,~ \left(im\, s_i, ker\, r_{2n-i}\right)_{IX} = 0.
\end{align*}
and:
\begin{align*}
&\alpha \in H_i(\partial T;k),~ \beta \in H_{2n-i}(I_{f_\bullet}X) \implies (u_i\ell_i\alpha, \beta)_{IX} = (-1)^{i}(v_{2n-i}\beta, \iota_i\alpha)_{L} \\
&\gamma \in im\, v_i,~ \delta \in H_{2n-i}(X- T^\circ;k) \implies (s_i\gamma, h_{2n-i}\delta)_{IX} = (\gamma, \delta)_L.     
\end{align*}
where $(-,-)_L$ denotes the intersection pairing associated to $D_L$, i.e.\ $(-,-)_L = D_L(-)(-)$.
\end{lem}
\begin{proof}
Suppose $\alpha \in H_i(\partial T;k)$ and $\beta \in H_{2n-i}(I_{f_\bullet}X)$.
Then:
\begin{align*}
(u_i\ell_i\alpha, \beta)_{IX} = D_{IX}(u_i\ell_i\alpha)(\beta) = (g^{2n-i}D_{f_\bullet}\ell_i\alpha)(\beta) =  (D_{f_\bullet}\ell_i\alpha)(g_{2n-i}\beta).
\end{align*}
Since exactness implies $ker\, g_{2n-i} = im\, h_{2n-i}$, and $im\, u_i = im\, u_i\ell_i$, this proves that $(im\,u_i, im\, h_{2n-i})_{IX} = 0$.
But let's go further with our computation of $(u_i\ell_i\alpha,\beta)_{IX}$; it is equal to:
\begin{align*}
(g^{2n-i}D_{f_\bullet}\ell_i\alpha)(\beta) &= (g^{2n-i}f^{2n-i-1}D_\partial\alpha)(\beta) \\
&= (v^{2n-i}\delta^{2n-i}D_\partial\alpha)(\beta) = D_L'(\iota_i\alpha)(v_{2n-i}\beta)\\
&= (-1)^{i(2n-i)} D_L(v_{2n-i}\beta)(\iota_i\alpha) = (-1)^i(v_{2n-i}\beta, \alpha)_L.
\end{align*}

Next suppose $\gamma \in im\, v_i$ and $\delta \in H_{2n-i}(X-T^\circ;k)$.
Then:
\begin{align*}
(s_i\gamma, h_{2n-i}\delta)_{IX} = D_{IX}(s_i\gamma)(h_{2n-i}\delta) &= (r^{2n-i}D_L\gamma)(h_{2n-i}\delta)\\
&= (D_L \gamma)(r_{2n-i}h_{2n-i}\delta)\\
& = (D_L\gamma)(\delta) = (\gamma, \delta)_{L}.
\end{align*} 
Note that it makes sense to evaluate $D_L \gamma$ on $im\, r_{2n-i} = coim\, h_{2n-i}$, since $\gamma$ belonging to $im\, v_i$ implies $D_\partial \gamma$ belongs to:
\begin{align*} im\, h^{2n-i} = (coim\, h_{2n-i})^* = \{\phi \in H^{2n-i}(X-T^\circ): \phi(ker\, h_{2n-i}) = 0 \}\end{align*} 
In the computations at the beginning of this paragraph, if we replace $h_{2n-i}\delta$ with an element of $ker\, r_{2n-i}$, it easily follows that $(im\, s_i , ker\, r_{2n-i})_{IX} = 0$.  
\end{proof}  

Lemma \hyp{intspacepair} informs us how to carry out the intersection space pairing, but we will seek an even finer decomposition of the homology vector spaces.
We say a family of sections $(r^\bullet, s_\bullet)$ is \bi{untwisted} if and only if each of the following hold:
\begin{enumerate}[(i)]
\item $im\, s_\bullet j_\bullet \subset im\, h_\bullet$ (observe that $j_\bullet = v_\bullet h_\bullet$ implies $im\, j_\bullet \subset im\, v_\bullet$, so $s_\bullet j_\bullet$ is in fact well-defined).
\item $ker\, r_\bullet \subset im\, s_\bullet$ and $(ker\, r_\bullet, im\, s_{2n-\bullet} j_{2n-\bullet})_{IX} = 0$.
\end{enumerate}
We will prove the existence of an untwisted family of sections, but first must establish a technical lemma.

\begin{lem}\label{Q}
Suppose $s_\bullet$ satisfies property (i) of untwisted, and $Q_\bullet$ is the maximal subspace of $im\, v_\bullet$ satisfying the vanishing:
\begin{align*}
\left(Q_\bullet, h_{2n-\bullet}^{-1}(im\, s_{2n-\bullet} j_{2n-\bullet})\right)_L = 0. 
\end{align*} 
Then $im\, v_\bullet = Q_\bullet \oplus im\, j_\bullet$.
\end{lem}
\begin{proof}
Fix $i \in \Z$.
First we check that $Q_i$ has trivial intersection with $im\, j_i$.
Suppose $j_i\alpha \in im\, j_i \cap Q_i$.
Then:
\begin{align*}
0 = \left(j_i\alpha, h_{2n-i}^{-1}(im\, s_{2n-i} j_{2n-i})\right)_L &= (D_Lj_i\alpha)\left(h_{2n-i}^{-1}(im\, s_{2n-i} j_{2n-i})\right)\\
&= (j^{2n-i}D_{L}'\alpha)\left(h_{2n-i}^{-1}(im\, s_{2n-i} j_{2n-i})\right).
\end{align*}
Since $j^{2n-i} = h^{2n-i}v^{2n-i}$ and $im\, s_{2n-i}j_{2n-i} \subset im\, h_{2n-i}$, this sequence of equalities continues into:
\begin{align*}
0 = (D_{L}'\alpha)(v_{2n-i}im\, s_{2n-i}j_{2n-i}) = (D_{L}'\alpha)(im\, j_{2n-i}) &\implies j^{2n-i}D'_L\alpha = 0\\
&\implies D_Lj_i\alpha = 0\\
&\implies j_i\alpha = 0.
\end{align*}

Having shown the intersection is trivial, we use dimension counting to complete the proof.
The duality isomorphism of long exact sequences of the pair $(X-T^\circ, \partial T)$ implies $D_L(im\, j_i) = im\, j^{2n-i}$.
Vanishing of the local duality obstruction implies $D_L(im\, v_i) = im\, h^{2n-i}$ (see the proof of Lemma \hyp{dualconstruction}); in particular this implies $(im\, v_i, ker\, h_{2n-i})_L = 0$. 
By construction, and the fact that $(im\, v_i, ker\, h_{2n-i})_L = 0$, the vector space $Q_i$ has dimension greater than or equal to the difference:
\begin{align*}
dim_k\, &im\, v_i - dim_k\, h_{2n-i}^{-1}(im\, s_{2n-i} j_{2n-i}) + dim_k\, ker\, h_{2n-i}\\
&= dim_k\, im\, v_i - dim_k\, im\, j_{2n-i}\\
&= dim_k\, im\, v_i - dim_k\, im\, j_i
\end{align*} 
Rearranged, this is:
\begin{align*}
dim_k\, Q_i + dim_k\, im\, j_i \geq dim_k\, im\, v_i,
\end{align*}
completing the proof.
\end{proof}

\begin{prop}\label{untwistedexist}
There always exists an untwisted family of sections $(r^\bullet, s_\bullet)$.
\end{prop} 
\begin{proof} 
First let's check that there exist sections $s_\bullet: im\, v_\bullet \rightarrow H_\bullet(I_{f_\bullet}X)$ of $v_\bullet$ such that $im\, s_\bullet j_\bullet \subset im\, h_\bullet$.
This is possible iff $v_\bullet$ maps $im\, h_\bullet$ \textit{onto} $im\, j_\bullet$ (because then we can construct a restricted section $s_\bullet|: im\, j_\bullet \rightarrow im\, h_\bullet \subset H_\bullet(I_{f_\bullet}X)$ which by choice of a basis for $im\, v_\bullet$ can be extended to a full section $s_\bullet$). 
But $im\, j_\bullet = im\, v_\bullet h_\bullet$, so $v_\bullet$ indeed maps $im\, h_\bullet$ surjectively onto $im\, j_\bullet$.   

Given these sections $s_\bullet$ satisfying (i), we next verify that there exist retractions $r_\bullet : H_\bullet(I_{f_\bullet}X) \rightarrow coim\, h_\bullet$ satisfying (ii).
Let $Q_\bullet$ be as in Lemma \hyp{Q}.
We will first show that $r_\bullet$ can be selected so that $ker\, r_\bullet \subset s_\bullet Q_\bullet$; this is possible iff $s_\bullet Q_\bullet + im\, h_\bullet = H_\bullet(I_{f_\bullet}X)$. 
We already know from Lemma \hyp{intspacepair} that $im\, s_\bullet + im\, u_\bullet = H_\bullet(I_{f_\bullet}X)$.
But $g_\bullet u_\bullet = 0$ implies $im\, u_\bullet \subset ker\, g_\bullet = im\, h_\bullet$, were we have used exactness in the last step.
So:
\begin{align*}
im\, s_\bullet + im\, h_\bullet = H_\bullet(I_{f_\bullet}X).
\end{align*}
If we can show $im\, s_\bullet = s_\bullet Q_\bullet + im\, s_\bullet j_\bullet$
then we will be done with selecting our $r_\bullet$, since $im\, s_\bullet j_\bullet \subset im\, h_\bullet$ by property (i) of being untwisted.
But this follows by applying $s_\bullet$ to the equality of Lemma \hyp{Q}.

It remains to verify that $ker\, r_\bullet \subset s_\bullet Q_\bullet$ satisfies property (ii).
This is clear, because Lemma \hyp{intspacepair} and the definition of $Q_\bullet$ imply:
\begin{align*}
(s_\bullet Q_\bullet, im\, s_{2n-\bullet}j_{2n-\bullet})_{IX} = \left(Q_\bullet, h_{2n-\bullet}^{-1}(im\, s_{2n-\bullet}j_{2n-\bullet})\right)_L = 0.
\end{align*}
\end{proof}

\setcounter{equation}{0}

\begin{lem}\label{untwisted}
Suppose $(r^\bullet, s_\bullet)$ is an untwisted family of sections and $i \in \Z$.
Then there exist further decompositions:
\begin{align*}
&im\, h_i = im\, u_i \oplus im\, s_ij_i\\
&im\, s_i = ker\, r_i \oplus im\, s_ij_i.
\end{align*}
\end{lem}
\begin{proof}
Let's begin with the decomposition of $im\, h_i$.
First we verify $im\, u_i \subset im\, h_i$: we have $g_iu_i = 0$, so $im\, u_i \subset ker\, g_i = im\, h_i$ where we have used exactness in the last equality.
Next, we know from Lemma \hyp{intspacepair} that $im\, s_ij_i$ has trivial intersection with $im\, u_i$.
Finally, we count dimension (using rank and nullity of maps):
\begin{align*}
rk\, j_i = rk\, v_ih_i \geq rk\, h_i - nul\, v_i &\implies rk\, h_i \leq nul\, v_i + rk\, j_i\\
&\implies dim_k\, im\, h_i \leq dim_k\, ker\, v_i + dim_k\, im\, j_i \\
&\implies dim_k\, im\, h_i \leq dim_k\, im\, u_i + dim_k\, im\, s_ij_i  
\end{align*}
where in the last step we have used exactness of a long exact sequence and injectivity of $s_i$.

Now let's approach the decomposition of $im\, s_i$.
First we verify $ker\, r_i \cap im\, s_ij_i = \langle 0 \rangle $: because the family of sections is untwisted, we know $im\, s_ij_i \subset im\, h_i$, but Lemma \hyp{intspacepair} implies $im\, h_i \cap ker\, r_i = \langle 0 \rangle$.
Finally for (2), we again count dimension:
\begin{align*}
dim_k\, im\, u_i + dim_k\, im\, s_i &= dim_k\, H_i(I_{f_\bullet}X) \\
&= dim_k\, im\, h_i + dim_k\, ker\, r_i\\
&= dim_k\, im\, u_i + dim_k\, im\, s_ij_i  + dim_k\, ker\, r_i .
\end{align*}
where we have used Lemma \hyp{intspacepair} for the first two inequalities, and (1) of this Lemma for the last.
We then obtain:
\begin{align*}
dim_k\, im\, s_i = dim_k\, im\, s_ij_i  + dim_k\, ker\, r_i
\end{align*} 
as desired.
\end{proof}

\subsection{Signature}
Suppose throughout this subsection that $(r^\bullet, s_\bullet)$ is a family of sections.
We would like to prove that, in the case that $n$ is even, an untwisted family of sections $(r^\bullet, s_\bullet)$ induces a symmetric pairing:
\begin{align*}
(-,-)_{IX}: H_n(I_{f_\bullet}X) \times H_n(I_{f_\bullet}X) \rightarrow k.
\end{align*} 
But first let's compute the signature  when $k = \Q$ if we \textit{assume} the induced pairing $(-,-)_{IX}$ is symmetric.
We will do so by comparing to the already existing symmetric pairing on $im\, j_n$:
\begin{align*}
(-,-)_{j}: im\, j_n \times im\, j_n \rightarrow \Q,~ (j_n\alpha, j_n\beta)_{j} = (j_n\alpha, \beta)_{L}  = (j_n\beta, \alpha)_{L};
\end{align*}
this signature is called the \bi{Novikov signature}, which is known (for example, \cite{S}) to equal the signature of the pairing:
\begin{align*}
IH_n(\widehat{X-T^\circ}; \Q) \times IH_n(\widehat{X-T^\circ}; \Q) \rightarrow \Q
\end{align*} 
where $\widehat{X-T^\circ}$ is the space $(X-T^\circ) \cup_{T} cone(T)$. 

\begin{thm}\label{sign}
Suppose $n$ is even, $k = \Q$, and $(r^\bullet, s_\bullet)$ is an untwisted family of sections that induces a symmetric intersection space pairing:
\begin{align*}
(-,-)_{IX}: H_n(I_{f_\bullet}X) \times H_n(I_{f_\bullet}X) \rightarrow \Q.
\end{align*}
Then the signature of $(-,-)_{IX}$ is equal to the Novikov signature.
\end{thm}
\begin{proof}
Remember that throughout this proof we are \textit{assuming} $(-,-)_{IX}$ is symmetric.
We will frequently use Lemmas \hyp{intspacepair} and \hyp{untwisted} in this proof.
Combined they give us the decompositions:
\begin{align*}
&H_n(I_{f_\bullet}X) = im\, u_n \oplus im\, s_nj_n \oplus ker\, r_n\\
&im\, h_n = im\, u_n \oplus im\, s_nj_n\\
&im\, s_n = im\, s_nj_n \oplus ker\, r_n.
\end{align*}
Let's observe how these components pair.
By Lemma \hyp{intspacepair}, the above decompositions, and symmetry of the intersection space pairing, we know that under the intersection space pairing:
\begin{align*}
&im\, u_n \perp (im\, u_n \oplus im\, s_nj_n)\\
&im\, s_nj_n \perp (im\, u_n \oplus ker\, r_n)\\
&ker\, r_n \perp (im\, s_nj_n \oplus ker\, r_n).  
\end{align*}
Therefore, in a basis that respects the direct sum decomposition $im\, u_n \oplus im\, s_nj_n \oplus ker\, r_n$, the pairing $(-,-)_{IX}$ is represented by a symmetric block matrix of the form:
$$M = \left(\begin{array}{ccc}
0 & 0 & A\\
0 & Y & 0\\
A^T & 0 & 0
\end{array}\right),$$
where $Y$ is the symmetric matrix associated to the restricted pairing:
\begin{align*}
(-,-)_{IX} : im\, s_n j_n \times im\, s_nj_n \rightarrow \Q.
\end{align*}  
Let $p(t)$ and $q(t)$ be the respective characteristic polynomials for $Y$ and $A$.
Linear algebra shows that the characteristic polynomial of the block matrix $M$ is the product:
\begin{align*}
q(t)p(t)q(-t).
\end{align*}
Therefore the signature of $M$ (the number of positive eigenvalues minus the number of negative eigenvalues) is equal to the signature of $Y$, i.e.\ the signature of the restricted pairing:
\begin{align*}
im\, s_nj_n \times im\, s_nj_n \rightarrow \Q.
\end{align*} 

It remains to prove that the signature of this restricted pairing is the same as the signature of $(-,-)_j$; this is established if we can prove that the following diagram commutes:
$$\begin{tikzcd}
im\, s_nj_n \times im\, s_nj_n \ar[r] & k\\
im\, j_n \times im\, j_n \ar[u, "s_n \times s_n", "\simeq" {anchor=south, rotate=270, inner sep=1mm}] \ar[r] & k. \ar[equal, u]
\end{tikzcd}$$

Let $j_n\alpha$ and $j_n\beta$ in $im\, j_n$ be given.
Since $im\, s_nj_n \subset im\, h_n$ by untwistedness of the family of sections, there exists $\gamma \in H_n(X-T^\circ)$ such that $h_n\gamma = s_nj_n\beta$.
By Lemma \hyp{intspacepair} we have:
\begin{align*}
(s_nj_n\alpha, s_nj_n\beta)_{IX} = (s_nj_n\alpha, h_n\gamma)_{IX} =  (j_n\alpha, \gamma)_\partial = (j_n\gamma, \alpha)_{\partial},  
\end{align*}
where symmetry of the $j$-pairing was used in the last step.
Since $j_n = v_nh_n$ this further becomes:
\begin{align*}
(s_nj_n\alpha, s_nj_n\beta)_{IX} = (v_nh_n\gamma, \alpha)_{\partial} = (v_ns_nj_n\beta, \alpha) = (j_n\beta, \alpha)_\partial = (j_n\alpha,j_n\beta)_{j}.
\end{align*}
This proves that the diagram in question commutes, and we are finished.
\end{proof}

Next we verify that the pairing $(-,-)_{IX}$ induced by an untwisted family is indeed symmetric.
If $V$ and $W$ are subspace of $H_n(I_{f_\bullet}X)$, we say that the pairing $(-,-)_{IX}$ is \bi{symmetric on $(V,W)$} if and only if for all $\alpha \in V$ and $\beta \in W$ we have:
\begin{align*}
(\alpha,\beta)_{IX} = (\beta, \alpha)_{IX}.
\end{align*}

\begin{prop}
Suppose $n$ is even and $(r^\bullet, s_\bullet)$ is an untwisted family of sections.
Then $(-,-)_{IX}$ is symmetric on $H_n(I_{f_\bullet}X)$.
\end{prop}
\begin{proof}

\textit{Symmetry on $(im\, u_n, im\, u_n)$}: By Lemmas \hyp{intspacepair} and \hyp{untwisted}, we know $(im\, u_n, im\, u_n)_{IX} = 0$.

\textit{Symmetry on $(im\, u_n, im\, s_nj_n)$}:
By the same Lemmas, we know $(im\, u_n, im\, s_nj_n)_{IX} = 0$.
We must then check $(im\, s_nj_n, im\, u_n)_{IX} = 0$.
Let $s_nj_n\alpha$ and $u_n\beta$ be given.
Since $im\, u_n \subset im\, h_n = ker\, v_n$, we can write $u_n\beta = h_n\gamma$ and we observe:
\begin{align*}
j_n\gamma = v_nh_n\gamma = v_nu_n\beta = 0.
\end{align*} 
Using Lemma \hyp{intspacepair} and symmetry of the $j$-pairing we compute:
\begin{align*}
(s_nj_n\alpha, u_n\beta)_{IX} = (s_nj_n\alpha, h_n\gamma)_{IX} = (j_n\alpha, \gamma)_{L} = (j_n\gamma, \alpha)_{L} = 0.
\end{align*}

\textit{Symmetry on $(im\, u_n, ker\, r_n)$.}
Let $u_n\ell_n\alpha \in im\, u_n\ell_n = im\, u_n$ and $s_n\beta \in ker\, r_n \subset im\, s_n$ be given.
Since $im\, u_n \subset im\, h_n$ by Lemma \hyp{untwisted}, there exists $\gamma$ such that $u_n\ell_n\alpha = h_n\gamma$.
By Lemma \hyp{intspacepair}:
\begin{align*}
(u_n\ell_n\alpha, s_n\beta)_{IX} = (v_ns_n\beta, \iota_n \alpha)_L = (\beta, \iota_n\alpha)_L
\end{align*}
and:
\begin{align*}
(s_n\beta, u_n\ell_n\alpha)_{IX} = (s_n\beta, h_n\gamma)_{IX} = (\beta, \gamma)_L
\end{align*}
Next note that $\beta \in im\, v_n$ implies $(\beta, -)_L \in im\, h^n$ vanishes on $ker\, h_n$.
So symmetry can be proven if $\iota_n \alpha - \gamma \in ker\, h_n$.
This follows from:
\begin{align*}
h_n\iota_n\alpha = u_n\ell_n\alpha = h_n\gamma. 
\end{align*}

\textit{Symmetry on $(im\, s_nj_n, im\, s_nj_n)$.} Let $s_nj_n\alpha, s_nj_n\beta \in im\, s_nj_n$.
The reasoning from the last part of the proof of Theorem \hyp{sign} did not rely on symmetry of $(-,-)_{IX}$ and shows:
\begin{align*}
(s_nj_n\alpha, s_nj_n\beta)_{IX} = (j_n\alpha, j_n\beta)_j 
\end{align*}   
which of course is symmetric. 

\textit{Symmetry on $(im\, s_nj_n, ker\, r_n)$.} By Lemma \hyp{intspacepair} we have the vanishing $(im\, s_nj_n, ker\, r_n)_{IX} = 0$.
By property (ii) of being untwisted, we have the reverse vanishing $(ker\, r_n, im\, s_nj_n)_{IX}$.
\end{proof}

\end{document}